\documentclass[11pt,oneside, letterpaper]{amsart}
\usepackage{calc,amssymb,amsthm,amsmath,ulem}
\usepackage{alltt}
\RequirePackage[dvipsnames,usenames]{color}

\input{kmacros3.sty}
\usepackage{mabliautoref}
\usepackage[left=1.25in,top=1.1in,right=1.25in,bottom=1.1in]{geometry}
\usepackage{hyperref}
\usepackage{graphicx}
\usepackage[all,cmtip]{xy}

\usepackage{colonequals}

\renewcommand{\myR}{{\bf R}}

\renewcommand{\O}{\sO}

\newcommand{\ie}{{\itshape i.e.} }
\newcommand{\roundup}[1]{\lceil #1 \rceil}
\newcommand{\rounddown}[1]{\lfloor #1 \rfloor}
\renewcommand{\to}[1][]{\xrightarrow{\ #1\ }}
\newcommand{\ot}[1][]{\xleftarrow{\ #1\ }}
\newcommand{\onto}[1][]{\protect{\xrightarrow{\ #1\ }\hspace{-0.8em}\rightarrow}}
\newcommand{\into}[1][]{\lhook \joinrel \xrightarrow{\ #1\ }}
\newcommand{\tauCohomology}{T}
\newcommand{\tr}{\operatorname{Tr}}
\newcommand{\dcx}{\omega^\mydot}%Dualizingcomplex
\newcommand{\dsh}{\omega}%Dualizingsheaf
\newcommand{\h}[1][]{\myH^{#1}}
\renewcommand{\sHom}{\sH\negthinspace om}
\newcommand{\Sch}{\operatorname{Sch}}

\theoremstyle{theorem}
\newtheorem*{Mainthm}{Main Theorem}
\newtheorem*{cor*}{Corollary}

\numberwithin{equation}{section}

\begin{document}

\title{$F$-singularities via alterations}
\author{Manuel Blickle, Karl Schwede, Kevin Tucker}

\address{Institut f\"ur Mathematik\\ Johannes Gutenberg-Universit\"at Mainz\\55099 Mainz, Germany}
\email{blicklem@uni-mainz.de}
\address{Department of Mathematics\\ The Pennsylvania State University\\ University Park, PA, 16802, USA}
\email{schwede@math.psu.edu}
\address{Department of Mathematics\\ University of Utah\\ Salt Lake City, UT, 84112, USA}
\email{kevtuck@math.utah.edu}

\thanks{The first author was partially supported by a DFG Heisenberg Fellowship and the DFG research grant SFB/TRR45}
\thanks{The second author was partially supported by an NSF Postdoctoral Fellowship \#0703505, the NSF grant \#1064485, the NSF FRG grant \#1265261, the NSF CAREER grant \#1252860 and a Sloan Fellowship}
\thanks{The third author was partially supported by an NSF Postdoctoral fellowship \#1004344 and NSF grant \#1303077.}

\maketitle

\begin{abstract}
We give characterizations of test ideals and $F$-rational singularities via (regular) alterations.  Formally, the descriptions are analogous to standard characterizations of multiplier ideals and rational singularities in characteristic zero via log resolutions.  Lastly, we establish Nadel-type vanishing theorems (up to finite maps) for test ideals, and further demonstrate how these vanishing theorems may be used to extend sections.
\end{abstract}

\section{Introduction}%\ok{--}{Ok}{--}

In this paper we define an ideal, in arbitrary equal characteristic, which coincides with the multiplier ideal over $\bC$, and coincides with the test ideal in characteristic $p > 0$.  This justifies the maxim: \begin{center}\emph{The test ideal and multiplier ideal are morally equivalent.}\end{center}
We state our main theorem.

\begin{Mainthm}[Theorem \ref{thm.MainThm}, Corollary \ref{cor.TestIdealConstructedForManyAlterations}, Theorem \ref{cor.MultIdealViaAlterations}]
\label{main.thm}
Suppose that $X$ is a normal algebraic variety over a perfect field and $\Delta$ is a $\bQ$-divisor on $X$ such that $K_X + \Delta$ is $\bQ$-Cartier.  %Set $K(X)$ to be the fraction field of $X$.
Consider the ideal
\[
J := \bigcap_{\pi \: Y \to X} \Image \Big( \pi_* \O_Y( \lceil K_Y - \pi^*(K_X + \Delta) \rceil ) \to[\tr_{\pi}] K(X)\Big) \, \, .
\]
Here the intersection runs over all generically finite proper separable maps $\pi \: Y \to X$ where $Y$ is regular (or equivalently just normal), and the map to the function field $K(X)$ is induced by the Grothendieck trace map $\tr_{\pi} \: \pi_* \omega_Y \to \omega_X$ (if $K_{Y} = \pi^{*}K_{X} + \Ram_{\pi}$ over the locus where $\pi$ is finite, then $\tr_{\pi} \: \pi_{*}\O_{Y}(K_{Y}) \to \O_{X}(K_{X})$ is induced by the field trace $\tr_{K(Y)/K(X)} \: K(Y) \to K(X)$). We obtain the following:
\begin{enumerate}
\item  If $X$ is of equal characteristic zero, then $J = \mJ(X; \Delta)$, the multiplier ideal of $(X, \Delta)$.
\item  If $X$ is of equal characteristic $p > 0$, then $J = \tau(X; \Delta)$, the test ideal of $(X, \Delta)$.
\end{enumerate}
Furthermore, in either case, the intersection defining $J$ stabilizes: in other words, there is always a generically finite separable proper map $\pi : Y \to  X$ with $Y$ regular such that $J = \Image \big( \pi_* \O_Y(K_Y - \pi^*(K_X + \Delta) ) \to[\tr_{\pi}] K(X)\big)$.
\end{Mainthm}

In fact, we prove a number of variants on the above theorem in further generality, \textit{i.e.} for various schemes other than varieties over a perfect field.

Of course, there are two different statements here.  In characteristic zero, this statement can be viewed as a generalization of the transformation rule for multiplier ideals under generically finite proper dominant maps, see \cite[Theorem 9.5.42]{LazarsfeldPositivity2} or \cite[Proposition~2.8]{EinMultiplierIdealsVanishing}.
In positive characteristic, a basic case of the theorem is the following characterization of $F$-rational singularities -- which is interesting in its own right.  Recall that an \emph{alteration} is a proper and generically finite map $\pi : Y \to X$, it is called a \emph{regular alteration} if $Y$ is a regular scheme \cite{de_jong_smoothness_1996}.
\begin{cor*}[\autoref{cor.CharOfFRational}, \textit{cf.} \cite{HunekeLyubeznikAbsoluteIntegralClosure, HochsterYaoUnpublished}]
 An $F$-finite ring $R$ of characteristic $p > 0$ is $F$-rational if and only if it is Cohen-Macaulay and for every alteration (equivalently every regular separable alteration if $R$ is of finite type over a perfect field, equivalently every finite dominant map with $Y$ normal) $f : Y \to \Spec R$, the map $f_* \omega_Y \to \omega_R$ is surjective.
\end{cor*}

The proof of this special case is in fact the key step in the proof of the main theorem in positive characteristic. The central ingredients in its proof are the argument of K. Smith \cite{SmithFRatImpliesRat} that $F$-rational singularities are pseudo-rational, and the work of C. Huneke and G. Lyubeznik on annihilating local cohomology using finite covers \cite{HunekeLyubeznikAbsoluteIntegralClosure} (\textit{cf.} \cite{HochsterHunekeInfiniteIntegralExtensionsAndBigCM,HochsterYaoUnpublished,sannai_galois_2011}).  The proof of the Main Theorem additionally utilizes transformation rules for test ideals under finite morphisms \cite{SchwedeTuckerTestIdealFiniteMaps}.

In this paper, we also give a transformation rule for test ideals under proper dominant (and in particular proper birational) maps between varieties of the same dimension.  More precisely, for any normal (but not necessarily proper) variety $Y$ and $\bQ$-divisor $\Gamma$, we define a canonical submodule $\tauCohomology^{0} (Y, \Gamma) \subseteq H^0(Y, \O_Y(\lceil K_Y + \Gamma \rceil))$.  We use this submodule to obtain a transformation rule for test ideals under alterations.

\begin{theorem*}[Theorem \ref{thm.TransformationOfTestIdealsUnderDominantMaps}]
Suppose that $\pi \: Y \to X = \Spec R$ is a proper dominant generically finite map of normal varieties over a perfect (or even $F$-finite) field of characteristic $p > 0$.  Further suppose that $\Delta$ is a $\bQ$-divisor on $X$ such that $K_X + \Delta$ is $\bQ$-Cartier.

Consider the canonically determined submodule (see Definition~\ref{def.TauCohomology})
\[
\tauCohomology^{0} (Y, -\pi^*(K_X + \Delta)) \subseteq H^{0}(Y, \O_Y(\lceil K_Y - \pi^*(K_X + \Delta))\rceil))
\]
of sections which are in the image of the trace map for any alteration of $Y$.
Then the global sections of $\tau(X; \Delta)$ coincide with the image of $\tauCohomology^{0} (Y, -\pi^*(K_X + \Delta))$ under the map
\[
H^{0}(Y, \O_Y(\lceil K_Y - \pi^*(K_X + \Delta)\rceil)) \to[\tr_{\pi}] K(X)
\]
which is induced by the trace $\tr_{\pi} \: \pi_* \omega_Y\to \omega_X$.
\end{theorem*}
\noindent We also prove a related transformation rule for multiplier ideals under arbitrary proper dominant maps in Theorem \ref{thm.TransformationRuleForMultProperDominant}.

Perhaps the most sorely missed tools in positive characteristic birational algebraic geometry (in comparison to characteristic zero) are vanishing theorems for cohomology.  Indeed,
Kodaira vanishing fails in positive characteristic \cite{raynaud_contre-exemple_1978}.  However, if $X$ is projective in characteristic $p > 0$ and $\sL$ is a ``positive'' line-bundle, cohomology classes $z \in H^i(X, \sL^{-1})$ can often be killed by considering their images in $H^i(Y, f^* \sL^{-1})$ for finite covers $f \: Y \to X$.  For example, if $i \geq 0$ and $\sL$ big and semi-ample, it was shown in \cite{BhattDerivedDirectSummand, BhattThesis} that there exists such a cover killing any cohomology class $\eta \in H^i(X, \sL^{-1})$ for $i < \dim X$ (\cf \cite{HochsterHunekeInfiniteIntegralExtensionsAndBigCM, SmithVanishingSingularitiesAndEffectiveBounds,SmithErratumVanishingSingularitiesAndEffectiveBounds}).  When we combine our main result with results from \cite{BhattThesis}, we obtain the following variant of a Nadel-type vanishing theorem in characteristic $p > 0$ (and a relative version).  Notably, we need not require a W2 lifting hypothesis.
\begin{theorem*}[\autoref{thm.NadelInCharP}]
Let $X$ be a normal proper algebraic variety of finite type over a perfect (or even $F$-finite) field of characteristic $p > 0$, $L$ a Cartier divisor, and $\Delta \geq 0$ a $\bQ$-divisor such that $K_{X} + \Delta$ is $\bQ$-Cartier.  Suppose that $L - (K_X + \Delta)$ is a big and semi-ample $\bQ$-divisor.  Then there exists a finite surjective map $f \: Y \to X$ such that:
\begin{enumerate}
\item  The natural map $f_* \O_Y(\lceil K_Y + f^*(L - K_X - \Delta) \rceil) \to \O_X( L )$, induced by the trace map, has image $\tau(X; \Delta) \tensor \O_X(L)$.
\item  The induced map on cohomology
\[
H^i(Y, \O_Y(\lceil K_Y + f^*(L - K_X - \Delta) \rceil) ) \to H^i(X, \tau(X; \Delta) \tensor \O_X(L))
\]
is zero for all $i > 0$.
\end{enumerate}
\end{theorem*}

Applying the vanishing theorem above, we obtain the following extension result.

\begin{theorem*} [Theorem \ref{thm.ExampleTheoremSurjectivity}]
Let $X$ be a normal algebraic variety which is proper over a perfect (or even $F$-finite) field of characteristic $p > 0$ and $D$ is a Cartier divisor on $X$.  Suppose that $\Delta$ is a $\bQ$-divisor having no components in common with $D$ and such that $K_X + \Delta$ is $\bQ$-Cartier.  Further suppose that $L$ is a Cartier divisor on $X$ such that $L - (K_X + D + \Delta)$ is big and semi-ample.  Consider the natural restriction map
\[
\gamma \: H^0(X, \O_X(\lceil L - \Delta \rceil) \to
H^0(D, \O_D(L|_D - \lfloor \Delta \rfloor|_D ) \, \, .
\]
Then
\[
\tauCohomology^0\big(D, L|_D - (K_D + \Delta|_D)\big) \subseteq \gamma\Big(\tauCohomology^0\big(X, D + L - (K_X + D + \Delta)\big)\Big)
\]
noting that $\tauCohomology^0(D, L|_D - (K_D + \Delta|_D)) \subseteq H^0(D, \O_D(\lceil L|_D - \lfloor \Delta \rfloor|_D )\rceil)$.
In particular, if $\tauCohomology^0(D, L|_D - (K_D + \Delta|_D)) \neq 0$, then
$H^0(X, \O_X(\lceil L - \Delta \rceil) \neq 0.$
\end{theorem*}

Finally, let us remark that many of the results contained herein can be extended to excellent (but not necessarily $F$-finite) \emph{local} rings with dualizing complexes; in fact, this is the setting of C. Huneke and G. Lyubeznik in \cite{HunekeLyubeznikAbsoluteIntegralClosure}. However, moving beyond the local case is then difficult essentially because we do not know the existence of test elements.  For this reason, and also because our inspiration comes largely from (projective) geometry, we restrict ourselves to the $F$-finite setting throughout (note that any scheme of finite type over a perfect field is automatically $F$-finite).

\vskip 6pt
\noindent {\it Acknowledgements:}
The authors would like to thank Bhargav Bhatt, Christopher Hacon, Mircea Musta{\c{t}}{\u{a}} and Karen Smith for valuable conversations.  The authors would also like to thank all the referees for many very useful comments.  Finally, the authors worked on this paper while visiting the Johannes Gutenberg-Universit\"at Mainz during the summers of 2010 and 2011.  These visits were funded by the SFB/TRR45 \emph{Periods, moduli, and the arithmetic of algebraic varieties}.

\section{The trace map, multiplier ideals and test ideals}

Multiplier ideals and test ideals are prominent tools in the study of singularities of algebraic varieties.  Later in this section, we will briefly review their constructions together with those of certain variants -- the multiplier and test module, respectively -- related to various notions of rational singularities. In doing so, we emphasize a viewpoint that relies heavily on the use of the trace map of Grothendieck-Serre duality. In fact, the whole paper (particularly \mbox{Sections~\ref{sec.NadelVanishing},~\ref{sec.TransformationRulesForTestIdeals},~and~\ref{sec:surj-cohom}}) relies on some of the more subtle properties of this theory.  First however we give a brief introduction to this theory necessary to understand the main results of this paper that does \emph{not} rely on any of these more subtle aspects.  %We recall the necessary statements and properties of dualizing complexes and the trace map below in some detail, to make the point that there is a common source for all of the maps we later consider in multiple seemingly unrelated cases (\textit{e.g.}~finite or birational maps). %The reader who is willing to accept this theory without further review may skip ahead to Example \ref{ex.traceForFinite} on page \pageref{ex.traceForFinite}.

{
%\color{blue}
\subsection{Maps derived from the trace map}\label{sec.Duality}%\ok{ok by me}{I'm fine with it}{Good enough.}

This section is designed to be a friendly and brief introduction to the trace map at the level we will apply it for our main theorem.  Therefore, in this subsection we only deal algebraic varieties of finite type over a perfect field (although everything can be immediately generalized to $F$-finite integral schemes).  More general versions will be discussed later in \autoref{sec.Duality}.  We will assume that the reader is familiar with canonical and dualizing modules at the level of \cite[Section 5.5]{KollarMori} and \cite[Chapter III, Section 7]{Hartshorne}.

Suppose that $\pi : Y \to X$ is a proper generically finite map of varieties of finite type over a field $k$.  A key tool in this paper is a \emph{trace map}
\[
    \tr_\pi \colon \pi_* \omega_Y \to \omega_X.
\]
Here $\omega_Y$ and $\omega_X$ denote suitable canonical modules on $Y$ and $X$ (which we assume exist). We will explain the origin of this map explicitly.  Since any generically finite map can be factored into a composition of a finite and proper birational map, it suffices to deal with these cases separately:
\begin{example}[Trace for proper birational morphism]
\label{ex:tracebirational}
Suppose that $\pi : Y \to X$ is a proper birational map between normal varieties.
In this case, the trace map $\tr_{\pi} \: \pi_*\omega_Y  \to \omega_{X}$ can be described in the following manner. Fix a canonical divisor $K_Y$ on $Y$ and set $K_X = \pi_* K_Y$ (in other words, recall by definition that $\omega_Y \cong \O_X(K_Y)$ and requiring that $\pi_* K_Y = K_X$ simply means that $K_X$ is the divisor on $X$ that agrees with $K_Y$ wherever $\pi$ is an isomorphism).  Then $\pi_* \O_Y(K_Y)$ is a torsion-free sheaf whose reflexification is just $\O_X(K_X)$, since $\pi$ is an isomorphism outside a codimension 2 set of $X$.  The trace map is simply the natural (reflexification) map $\pi_* \O_Y(K_Y) \hookrightarrow \O_X(K_X)$.
\end{example}

\begin{example}[Trace for finite morphism]
\label{ex.traceForFinite}
Suppose that $\pi \: Y \to X$ is a finite surjective map of varieties.  The trace map $\tr_{\pi} \: \pi_* \omega_Y \to \omega_X$ is then identified with the evaluation-at-1 map, $\pi_{*} \omega_{Y} \colonequals \sHom_{\O_X}(\pi_* \O_Y, \omega_X) \to \omega_X$ (the neophyte reader should take on faith that $\pi_{*} \omega_{Y} \cong \sHom_{\O_X}(\pi_* \O_Y, \omega_X)$, or see \autoref{sec.Duality} and \autoref{rem.FUpperShrieckConvenient} for additional discussion).
Assuming additionally that $\pi : Y \to X$ is a finite \emph{separable} map of normal varieties with ramification divisor $\Ram_{\pi}$, we fix a canonical divisor $K_X$ on $X$ and set $K_Y = \pi^* K_X + \Ram_{\pi}$.  Then the field-trace map
\[
\tr_{K(Y)/K(X)} : K(Y) \to K(X)
\]
restricts to a map $\pi_* \O_Y(K_Y) \to \O_X(K_X)$ which can be identified with the Grothendieck trace map (\cf \cite{SchwedeTuckerTestIdealFiniteMaps}).
\end{example}

Below in subsection \ref{sec.Duality} we will explain that this construction of a trace map is just an instance of a much more general theory contained in Grothendieck-Serre duality.
%This will also show that the trace for a generically finite map does not depend on the choice of the factorization.
We do not need this generality for our main theorem however.% (a fact that is rather cumbersome to check using the above construction)

We now mention two key properties that we will use repeatedly in this basic context.

\begin{lemma}[Compatibilities of the trace map]
\label{lem.CompatOfTraceBasic}
Suppose that $\pi : Y \to X$ is a proper generically finite dominant morphism between varieties (or integral schemes).  Fix $\Tr_{\pi} : \pi_* \omega_Y \to \omega_X$ to be the trace map as above.
\begin{itemize}
\item[(a)]  If additionally, $\rho : Z \to Y$ is another proper generically finite dominant morphism and $\Tr_{\rho} : \rho_* \omega_Z \to \omega_Y$ is the associated trace map, then $\Tr_{\pi} \circ (\pi_* \Tr_{\rho}) = \Tr_{\pi \circ \rho}$.
\item[(b)]  Additionally, if $U \subseteq X$ is open and $W = \pi^{-1}(U)$, then $\tr_{\pi|_W} = \tr_{\pi}|_{U}$ (here $\tr_{\pi} : \pi_* \omega_Y \to \omega_X$ is a map of sheaves on $X$ and so can be restricted to an open set).  In other words, the trace map is compatible with open immersions.
\end{itemize}
\end{lemma}
\begin{proof}
These properties follow directly from the definition given.
\end{proof}

Because much of our paper is devoted to studying singularities defined by Frobenius, utilizing \autoref{lem.CompatOfTraceBasic} we specialize \autoref{ex.traceForFinite} to the case where $\pi$ is the Frobenius.  %First we point out the following subtlety, for schemes not of finite type over a field, there can be more than one canonical choice of canonical module for a given $X$.  In this case, he canonical modules $\omega_Y$ and $\omega_X$ mentioned above have to be chosen with some care, and in particular, compatibly.  In the case that the map $\pi : Y \to X$ is the Frobenius map $F : X \to X$, we additionally require that $\sHom_{\O_X}(F_* \O_X, \omega_X) \cong F_* \omega_X$.  Again, this is automatic for any scheme of finite type over a field and we won't dwell on it here.  However, this subtlety will be discussed more below in \autoref{ex:TraceForFrobenius}.

\begin{example}[Trace of Frobenius]
\label{ex:TraceForFrobenius}
Suppose that $X$ is a variety of finite type over a perfect field of characteristic $p > 0$.  Then consider the absolute Frobenius map $F : X \to X$, this map is not a map of varieties over $k$, but it is still a map of schemes.  Using the fact (cf. Example \ref{ex:TraceForFrobeniusGeneral}) that $\sHom_{\O_X}(F_* \O_X, \omega_X) \cong F_* \omega_X$, and applying \autoref{ex.traceForFinite}, we obtain the evaluation-at-1 trace map,
\[
\Tr_{F} : F_* \omega_X \to \omega_X.
\]
%as described above.
Because of the importance of this map in what follows, we will use the notation $\Phi_X$ to denote $\Tr_{F}$.  As an endomorphism of $X$ one can compose the Frobenius with itself
%\[
%F^e : \underbrace{X \xrightarrow{F} X \xrightarrow{F} X \to \dots \to X}_{\text{$e$-times}}
%\]
and obtain the $e$-iterated Frobenius $F^e$.  It follows from \autoref{lem.CompatOfTraceBasic}(a) that $\Tr_{F^e}$ then coincides with the composition of $\Tr_{F}$ with itself $e$-times (appropriately pushed forward).  Because of this, we use $\Phi_X^e$ to denote $\Tr_{F^e}$.
\end{example}

%We will also often and repeatedly incorporate ($\bQ$-)divisors into the construction of our trace map as follows.
Now we come to a compatibility statement for images of trace maps that will be absolutely crucial later in the paper.  This is essentially the dual statement to a key observation from \cite{SmithFRatImpliesRat}.  We will generalize this later in \autoref{prop.EasyContainmentgeneral} and also in the proof of \autoref{prop.EasyContainmentViaGeneralMaps}.
\begin{proposition}\label{prop.EasyContainmentBase}
If $\pi \: Y \to X$ is a proper dominant generically finite map of varieties, then the image of the trace map
\[
J_\pi \colonequals \tr_\pi(\pi_*\omega_Y) \subseteq \omega_X
\]
satisfies $\Phi_{X}(F_*J_\pi) \subseteq J_\pi$.
\end{proposition}
\begin{proof}
Consider the commutative diagram:
\[
\xymatrix{
Y \ar[d]_{\pi} \ar[r]^{F} & Y \ar[d]^{\pi}\\
X \ar[r]_F & X
}
\]
where the horizontal maps are the Frobenius on $X$ and $Y$ respectively.
It follows from \autoref{lem.CompatOfTraceBasic}(a) that there is a commutative diagram
\[
\xymatrix@C=50pt{
F_* \pi_* \omega_Y \ar[d]_{F_* \Tr_{\pi}} \ar[r]^{\pi_* \Phi_Y} & \pi_* \omega_Y \ar[d]^{\Tr_{\pi}}\\
F_* \omega_X \ar[r]_{\Phi_X} & \omega_X.
}
\]
The claimed result follows immediately.
\end{proof}
%This result will be generalized in two ways.  Immediately below in \autoref{prop.EasyContainmentgeneral}, it will be generalized to arbitrary proper maps between schemes of possibly different dimensions (also incorporating the full dualizing complex).  %Later in Proposition \autoref{prop.EasyContainmentViaAlteration}, this will be generalized to include auxiliary $\bQ$-divisors.

\subsection{Pairs}
\label{sec.Pairs}

The next step is to extend the trace map to incorporate divisors.  Suppose that $X$ is a normal integral scheme.
A \emph{$\bQ$-divisor} $\Gamma$ on $X$ is a formal linear combination of prime Weil divisors with coefficients in $\bQ$. Writing $\Gamma = \sum b_i B_i$ where the $B_i$ are distinct prime divisors, we use $\roundup{\Gamma} = \sum \roundup{b_i} B_i$  and $\lfloor \Gamma \rfloor = \sum \rounddown{b_i} B_i$ to denote the round up and round down of $\Gamma$, respectively. We say $\Gamma$ is \emph{$\bQ$-Cartier} if there exists an integer $n > 0$ such that $n \Gamma$ is an integral (\textit{i.e.} having integer coefficients) Cartier divisor, and the smallest such $n$ is called the \emph{index} of $\Gamma$.  An integral divisor $K_X$ with $\O_X(K_X) \cong \omega_X$ is called a \emph{canonical divisor}. %son~$X$.

\begin{definition}
A \emph{pair $(X, \Delta)$} is the combined  data of a normal integral scheme $X$ together with a $\bQ$-divisor $\Delta$ on $X$.
The pair $(X, \Delta)$ is called \emph{log-$\bQ$-Gorenstein} if $K_X + \Delta$ is $\bQ$-Cartier.
\end{definition}

We emphasize that log-$\bQ$-Gorenstein pairs need \emph{not} be Cohen-Macaulay.

\begin{convention}
\label{rem.HowToEmbedOmegaXKX}
For $X$ normal and integral let $\Delta$ be a $\Q$-divisor on $X$.
The choice of a rational section $s \in \omega_X$ gives a canonical divisor $K_X=K_{X,s} = \operatorname{div} s$ and also a map $\omega_X \subseteq \omega_X \tensor K(X) \to[s \mapsto 1] K(X)$. Then the image of the inclusion $\omega_X(-K_{X,s}-\lfloor\Delta\rfloor) \into[s \mapsto 1] K(X)$ is the subsheaf $\O_X(-\lfloor \Delta \rfloor) \subseteq K(X)$. Note the image is independent of the choice of $s$ but the inclusion maps for different sections may differ by multiplication with a unit of $\O_X$.

Hence, every $\O_{X}$-submodule of $\omega_X(-\lfloor K_X + \Delta \rfloor)$ corresponds uniquely to an \mbox{$\O_{X}$-submodule} of $ \O_{X}(-\lfloor \Delta \rfloor )$ (or even $\subseteq \O_X$ when $\Delta$ is effective).  As such, we have chosen to accept certain abuses of notation in order to identify such submodules.  For example, we may write $\omega_X(-\lfloor K_X + \Delta \rfloor) \subseteq K(X)$ (or $\subseteq \O_X$ when $\Delta$ is effective);  however, while it is canonical as a subset (and equals $\O_{X}(-\lfloor \Delta \rfloor)$), the actual inclusion map involves the choice of a section (and is well defined only up to a multiplication by a unit of $\O_{X}$).
\end{convention}

We now state a result incorporating divisors into the trace map. % We first state the result in the case of an alteration since this is the situation we are primarily interested in.

\begin{proposition}
\label{prop.TraceInSomeGenerality}
Suppose that $\pi: Y \to X$ is a proper dominant generically finite morphism between normal varieties, and let $\Delta$ be a $\Q$-divisor on $X$ such that  $K_X+\Delta$ is $\Q$-Cartier. Then the trace map of $\pi$ induces a non-zero map
\[
   \tr_{\pi} \: \pi_*\omega_Y(-\lfloor \pi^*(K_X+\Delta) \rfloor) \to \omega_X(-\lfloor K_X+\Delta \rfloor) \, \, .
\]
\end{proposition}
\begin{proof}
This result is simple based upon Examples \ref{ex:tracebirational} and \ref{ex.traceForFinite} and so we leave it to the reader to verify.  We carefully prove a more general result in Propositions \ref{thm.TraceNonzero} \ref{prop.TraceInGenerality} below.
\end{proof}

}

\subsection{Duality and the trace map}
\label{sec.Duality}%\ok{ok by me}{I'm fine with it}{Good enough.}

This section restates the results of the previous section in the more general language of dualizing complexes.
While these results will be important for generalizations of our main theorem and for some of the Kodaira-type vanishing theorems, they are not needed for the main result stated in the introduction.  Therefore, we invite the reader to skip the next section and instead read ahead to \autoref{sec.MultiplierRational}.

From now on we assume that all schemes $X$ are Noetherian, excellent, separated and \emph{possess a dualizing complex} $\dcx_X$. This is a relatively mild condition, since, for example, all Noetherian schemes that are of finite type over a local Gorenstein ring of finite Krull dimension have a dualizing complex \cite[Chapter V, Section 10]{HartshorneResidues}. By definition, \cite[Chapter V, Section 2]{HartshorneResidues}, a \emph{dualizing complex on $X$} is an object in $D^b_{\coherent}(X)$ which has finite injective dimension and such that the canonical map $\O_X \to \myR\sHom_X(\dcx_X, \dcx_X)$ is an isomorphism in $D^b_{\coherent}(X)$.

Since dualizing complexes are defined by properties in the derived category, they are only unique up to quasi-isomorphism. But even worse, if $\dcx$ is a dualizing complex, then so is $\dcx \tensor \sL[n]$ for any integer $n$ and line-bundle $\sL$. But this is all the ambiguity there is for a connected scheme: if $\Omega^\mydot$ is another dualizing complex then there is a unique line-bundle $\sL$ and a unique shift $n$ such that $\Omega^\mydot$ is quasi-isomorphic to $\dcx \tensor \sL[n]$, see \cite[Theorem V.3.1]{HartshorneResidues}. The ambiguity with respect to shift is the least serious. For this, we say that a dualizing complex on an integral scheme (or a local scheme) is said to be \emph{normalized} if the first non-zero cohomology of $\dcx_X$ lies in degree $(-\dim X)$.

\newcommand{\upS}{\mathrm{S}}
A \emph{canonical module} $\dsh_X$ on a reduced and connected scheme $X$ is a coherent \mbox{$\O_X$-module} that agrees with the first non-zero cohomology of a dualizing complex $\dcx_X$. In particular, for a normalized dualizing complex $\dcx_X$, its $(-\dim X)$-th cohomology $\dsh_X \colonequals \h[-\dim X] \dcx_X$ is a canonical module; since it is the first non-zero cohomology there is a natural inclusion $\dsh_X[\dim X] \into \dcx_X$. If $X$ is $\upS_1$, \ie satisfies Serre's first condition, then $\omega_X$ is $\upS_2$ by \cite[Lemma 1.3]{HartshonreGeneralizedDivisorsAndBiliaison}.  Also see \cite[Corollary 5.69]{KollarMori} where it is shown that any (quasi-)projective scheme over a field has an $\upS_2$ canonical module.  We also note that if $X$ is integral, then $\omega_X$ can be taken to be the $\upS_2$-module agreeing with the dualizing complex on the Cohen-Macaulay locus of $X$.

We shall make extensive use of the trace map from Grothendieck-Serre duality, see \cite{ConradGDualityAndBaseChange,HartshorneResidues}.  For $S$ a base scheme, Noetherian, excellent, and separated, we consider the category $\Sch_{S}$ of $S$-schemes (essentially) of finite type over $S$, with $S$-morphisms between them. We assume as before that $S$ has a dualizing complex $\dcx_S$. Then Grothendieck duality theory provides us with a functor $f^!$ for every $S$-morphism $f\: Y \to X$ with the properties
\begin{enumerate}
\item $(\blank)^!$ is compatible with composition, \ie if $g \: Z \to Y$ is a further $S$-morphism, then there is a natural isomorphism of functors $(f \circ g)^! \cong g^! \circ f^!$ which is compatible with triple composition.
\item If $f$ is of finite type, and $\dcx_X$ is a dualizing complex on $X$ then $f^!\dcx_X$ is a dualizing complex on $Y$.  If additionally $f$ is a dominant morphism of integral schemes and $\dcx_{X}$ is normalized, then $f^{!}\dcx_{X}$ is also normalized.
\item If $f$ is a finite map, then $f^!( \blank ) = \myR\sHom_{\O_X}^{\mydot} (f_*\O_Y, \blank )$ viewed as a complex of \mbox{$\O_Y$-modules}. Note that the right hand side is defined for \emph{any} finite morphism, not just for an $S$-morphism.
\end{enumerate}
Therefore we may define for each $S$-scheme $X$ with structural map $\pi_X\: X \to S$ the dualizing complex $\dcx_X \colonequals \pi_X^!\dcx_S$.
  After this choice of dualizing complexes on $\Sch_S$, the compatibility with composition now immediately implies that for any $S$-morphism $f \: Y \to X$ we have a canonical isomorphism $f^!\dcx_X \cong \dcx_Y$.

  \begin{remark}
  In the remainder of the paper, the base scheme $S$ will typically either be a field, or it will be the scheme $X$ we are interested in.  Note that in either case the absolute Frobenius map $F : X \to X$ is not a map of $S$-schemes with the obvious choice of (the same) structural maps.  However, using the composition $F : X \to X \to S$, we do obtain a new $S$-scheme structure for $X$ and so we view $F : X \to X$ as a map of different $S$-schemes.
  \end{remark}

A key point in the construction of $(\blank)^!$ is that for $f \: Y \to X$ proper there is a natural transformation of functors
\[
   \myR f_*f^! \to \id_X
\]
called the trace map which induces a natural isomorphism of functors in the derived category
\[
    \myR f_* \myR\sHom_Y(\sM^\mydot, f^!\sN^\mydot)\to \myR \sHom_Y(\myR f_*\sM^\mydot, \sN^\mydot)
\]
for any bounded above complex of quasi-coherent $\O_Y$-modules $\sM^\mydot$ and bounded below complex of coherent $\O_X$-modules $\sN^\mydot$. This statement, which expresses that $f^!$ is right adjoint to $\myR f_*$ for proper $f$, is the duality theorem in its general form.

Applying trace map to the dualizing complex $\dcx_X$ we obtain
\begin{equation}\label{eq.trace}
\tr_{f^{\mydot}}: \myR f_*\dcx_Y \cong \myR f_*f^!\dcx_X \to \dcx_X
\end{equation}
which we also refer to as the trace map and denote by $\tr_{f^{\mydot}}$.  Equivalently, by the duality theorem, the trace map is Grothendieck-Serre dual to the corresponding map of structure sheaves $f \: \O_{X} \to \myR f_{*}\O_{Y}$.
The key properties of the trace relevant for us are
\begin{enumerate}
\item Compatibility with composition, \textit{i.e.} if $g \: Z \to Y$ is another proper $S$-morphism then $\tr_{(f \circ g)^{\mydot}} = \tr_{f^{\mydot}} \circ \myR f_* \tr_{g^{\mydot}}$.
\item The trace map is compatible with certain base changes. In general this is a difficult and subtle issue (see \cite{ConradGDualityAndBaseChange}); however, we will only need this for open inclusions $U \subseteq X$ and localization at a point, where it is not problematic.
\item In the case that $f$ is finite, $\tr_{f^{\mydot}}$ is locally given by evaluation at $1$,
\[
\tr_{f^{\mydot}}: \myR f_* f^! \dcx_X = \myR f_* \myR \sHom_{\O_X}^{\mydot}(f_* \O_Y, \dcx_X) \to \dcx_X
\]
\end{enumerate}

For a proper dominant morphism $\pi \: Y \to X$ of integral schemes, the trace gives rise to maps on canonical modules as well (not just dualizing complexes).
Taking the ($-\dim X$)-th cohomology of $\tr_{\pi^{\mydot}} \: \myR \pi_* \dcx_Y \to \dcx_X$ gives
\begin{equation}
\label{eq:tracedualizing}
\tr_{\pi^{\mydot}} \: \myH^{-\dim X}(\myR \pi_*\dcx_Y) \to \myH^{-\dim X}\dcx_X  =: \omega_{X}
\end{equation}
which we will also denote by $\tr_{\pi^{\mydot}}$.
Further composing with the inclusion $\omega_{Y}[\dim {Y}] \to \dcx_{Y}$ then gives
\begin{equation}\label{eq.TraceCanonical}
    \tr_\pi \: \myH^{\dim Y - \dim X}\myR \pi_*\omega_Y \to \myH^{-\dim X}\myR \pi_*\dcx_Y \to[\tr_{\pi^{\mydot}}] \omega_X
\end{equation}
which we also refer to as the trace map and  now denote by $\tr_\pi$. Note that the above construction remains compatible with localization on the base scheme, which we make use of below in showing under mild assumptions that this trace map is non-zero.

\begin{proposition}\label{thm.TraceNonzero}
If $\pi \: Y \to X$ is a proper dominant morphism of integral schemes, then the trace map
\[
    \tr_\pi \: \myH^{\dim Y-\dim X}\myR\pi_*\omega_Y \to \omega_X
\]
is non-zero.
\end{proposition}
\begin{proof}
To show the statement we may assume that $X = \Spec K$ for $K$ a field. The map $\mathbb{H}^0(Y,\dcx_Y) = \myH^0 \myR\pi_*\dcx_Y \to \dcx_K\cong K$ is non-zero as it is Grothendieck-Serre dual to the natural inclusion $K \to H^0(Y,\O_Y)$.
%, which is a finite extension field of $K$.
%THIS ABOVE CLAUSE IS NOT TRUE AS $Y$ IS NOT ASSUMED INTEGRAL OR EVEN REDUCED
Hence it is sufficient to show that the map $\myH^{\dim Y}(\myR\pi_* \omega_Y) \to \myH^0(\myR\pi_* \dcx_Y)$ is surjective. If $Y$ is Cohen-Macaulay, \ie $\omega_Y[\dim Y]=\dcx_Y$, we are done. More generally consider the hypercohomology spectral sequence $E_2=H^i(Y,\myH^j\dcx_Y) \Rightarrow \mathbb{H}^{i+j}(Y,\dcx_Y)$. The only terms contributing to $\mathbb{H}^0(Y,\dcx_Y)$ are $H^i(Y,\myH^j\dcx_Y)$ with $i+j=0$. In the next lemma, it is shown that $\dim \supp \myH^j\dcx_Y < -j$ for $j > -\dim {Y}$, which hence implies that $H^{\dim Y}(Y,\omega_Y) = \myH^{\dim Y}(\myR\pi_* \omega_Y)$ is the only non-vanishing term among them and thus surjects onto $\myH^0(\myR\pi_* \dcx_Y)$ as claimed.
\end{proof}

\begin{lemma}
Let $(R,\mathfrak{m})$ be an equidimensional local $\textnormal{S}_1$ ring of dimension $n$ with normalized dualizing complex $\dcx_R$. Then
\[
    \dim \supp \myH^{-j} \dcx_R < j
\]
for $j < n$.
\end{lemma}
\begin{proof}
\newcommand{\frp}{\mathfrak{p}}
By local duality $\myH^0 \dcx_R$ is Matlis dual to $\Gamma_\mathfrak{m}(R)$ which by the $\textnormal{S}_1$ condition is zero. This shows the lemma for $j=0$, and hence in particular for $n=1$, so that we may proceed by induction on $n$.

Assuming that $1 \leq j < n$, if $\dim \supp \myH^{-j}\dcx_R = 0$ we are done.  Otherwise, we have $c = \dim \supp \myH^{-j}\dcx_R > 0$ and can take $\frp \neq \bm$ to be a minimal prime in the support of $\myH^{-j}\dcx_R$ with $\dim R/\frp = c$.  Since $\dim R_{\frp} = n - c < n$, we have that $(\dcx_{R})_{\frp}[-c]$ is a normalized dualizing complex for $R_{\frp}$. Thus, by the induction hypothesis it follows
\[
0 = \dim \supp \myH^{-j}(\dcx_R)_\frp = \dim \supp \myH^{-j+c}(\dcx_R)_\frp[-c] < j-c
\]
so that $c = \dim \supp \myH^{-j}\dcx_R < j$ as desired.
% Assuming that for some $1 \leq j < n$ that $c = \dim \supp \myH^{-j}\dcx_R \geq j$, let $\frp$ be a minimal prime of the support of $\myH^{-j}\dcx_R$. Hence $\myH^{-j}(\dcx_R)_\frp \neq 0$ and $(\dcx_R)_\frp[-c]$ is a normalized dualizing complex for $R_\frp$ since $\dim R_\frp = n-c$. Hence
% \[
%     \dim \supp \myH^{-j}(\dcx_R)_\frp = \dim \supp \myH^{-j+c}(\dcx_R)_\frp[-c] < j-c \leq 0
% \]
% by applying the induction hypothesis to the $(n-c)$-dimensional ring $R_\frp$, noting that $\dim R_\frp \neq j-c \leq 0$. But this is a contradiction since $\myH^{-j}(\dcx_R)_\frp$ was nonzero by construction.
\end{proof}

We address now a particularly subtle issue surrounding the upper shriek functor and Frobenius.

\begin{example}[Trace of Frobenius and behavior of dualizing complexes]
\label{ex:TraceForFrobeniusGeneral}
A particularly important setting in this paper is that of a scheme $X$ essentially of finite type over an \emph{$F$-finite} base scheme $S$ of positive characteristic $p$ (\textit{e.g.} a perfect
 field of characteristic $p > 0$). This means simply that the (absolute) Frobenius or $p$-th power map $F \: X \to X$ is a finite morphism, and hence proper. Thus we have the ``evaluation at 1" trace map
\[
   \tr_{F^{\mydot}} \: F_*F^! \dcx_X \cong F_*\myR\sHom_X(F_*\O_X,\dcx_X) \to \dcx_X \, \, .
\]
However, note that since the Frobenius $F$ is generally not an $S$-morphism we do not have, a priori, that $F^!\dcx_X \cong \dcx_X$. This needs an additional assumption, namely that this property holds for the base scheme $S$. Using a fixed isomorphism $F^!\dcx_S \cong \dcx_S$, the compatibility of $(\blank)^!$ with composition in the commutative diagram
\[
\xymatrix{
X \ar[d]_{\pi} \ar[r]^F & X \ar[d]^{\pi} \\
S \ar[r]_F & S
}
\]
shows that indeed
\[
    F^! \dcx_X \cong F^! \pi^! \dcx_S \cong \pi^! F^! \dcx_S \cong \pi^! \dcx_S \cong \dcx_X \, \, .
\]
%Now, applying example \ref{ex.traceForFinite} to the specific case of Frobenius yields the map
%\[
%\Phi_X = \tr_{F}: F_* \omega_X(-\lfloor p(K_X + \Delta) \rfloor) \to \omega_X(\lfloor \Delta \rfloor)
%\]
%which we denote by a special letter $\Phi_X$ as it plays a pivotal role in this paper.
%Again, we need make no assumption that $K_X + \Delta$ is $\bQ$-Cartier in order to define this map.
\end{example}

\begin{convention}
  For simplicity, we will always assume that all our base schemes $S$ of positive characteristic $p$ are $F$-finite and
  satisfy $F^! \dcx_S \cong \dcx_S$ for the given choice of dualizing
  complex $\dcx_S$.  This is automatic if $S$ is the spectrum of a local ring (\textit{e.g.} a field).
\end{convention}

\begin{remark}
\label{rem.FUpperShrieckConvenient}
The assumption $F^!\dcx_S \cong \dcx_S$ is convenient but could nonetheless be avoided.  Notice that regardless, $F^! \dcx_S$ is a dualizing complex, and so it already agrees with $\dcx_S$ up to a shift and up to tensoring with an invertible sheaf.  It is easy to see that the shift is zero since $F$ is a finite map,  thus we have $F^! \dcx_S \cong \dcx_S \tensor_{\O_S} \sL$ for some invertible sheaf $\sL$.  One option would be to carefully keep track of $\sL$ throughout all constructions and arguments -- this we chose to avoid.  In any case, if one is willing to work locally over the base $S$ (as is the case with most of our main theorems) one may always assume that $F^! \dcx_S = \dcx_S$ simply by restricting to charts where $\sL$ is trivial.
\end{remark}

\autoref{prop.TraceInSomeGenerality} and \autoref{thm.TraceNonzero} can be combined as follows.

\begin{proposition}
\label{prop.TraceInGenerality}
Suppose that $\pi: Y \to X$ is a proper dominant morphism between normal integral schemes, and let $\Delta$ be a $\Q$-divisor on $X$ such that  $K_X+\Delta$ is $\Q$-Cartier. Then the trace map of $\pi$ induces a non-zero map
\[
   \tr_{\pi} \: \myH^{\dim Y-\dim X}\myR\pi_*\omega_Y(-\lfloor \pi^*(K_X+\Delta) \rfloor) \to \omega_X(-\lfloor K_X+\Delta \rfloor) \, \, .
\]
 Similarly, if additionally $\pi^*(K_X + \Delta)$ is a Cartier divisor, then we have another non-zero map
\[
    \Tr_{\pi^{\mydot}} : \myH^{-\dim X}\myR\pi_*\omega_Y^{\mydot}(-\lfloor \pi^*(K_X+\Delta) \rfloor) \to \omega_X(-\lfloor K_X+\Delta \rfloor) \, \, .
\]
\end{proposition}
\begin{proof}
The difficulty here lies in that $\lfloor K_X + \Delta \rfloor$ need not be $\Q$-Cartier. To overcome this, let  $U \into[i] X$ be the regular locus of $X$, then we have the trace map
\[
    \myH^{\dim {Y }-\dim X}\myR\pi_*\omega_{\pi^{-1}(U)} \to[\tr_{\pi^{\mydot}}] \omega_U
\]
which is just the restriction of \eqref{eq.TraceCanonical} to $U$.  % We now compose with the map $\omega_{\pi^{-1}U}[\dim Y] \to \dcx_{\pi^{-1}U}$ to obtain the map
% \[
% \myH^{\dim Y -\dim X}\myR\pi_* \omega_{\pi^{-1}U} \to \omega_U.
% \]
Tensoring this map by the invertible sheaf $\O_{U}(- \lfloor (K_X + \Delta) \rfloor)$, we have by the projection formula
\[
 \myH^{\dim {Y}-\dim X} \myR \pi_* \omega_{\pi^{-1}(U)}(-\pi^* \lfloor (K_X + \Delta) \rfloor) \to \omega_U(-\lfloor (K_X + \Delta) \rfloor ) \, \, .
\]
Since $-\pi^* \lfloor K_X + \Delta \rfloor \geq -\lfloor \pi^*(K_X + \Delta) \rfloor$ on $U$ (where again $\lfloor K_{X} + \Delta \rfloor$ is Cartier), we have an induced map
\begin{equation}
\label{eqn.twisted}
 \myH^{\dim Y-\dim X} \myR \pi_* \omega_{\pi^{-1}(U)}(- \lfloor \pi^*(K_X + \Delta) \rfloor) \to \omega_U(-\lfloor (K_X + \Delta) \rfloor ) \, \, .
\end{equation}
Applying $i_*(\blank)$ to \eqref{eqn.twisted} and composing with the restriction map
\[
 \myH^{\dim Y-\dim X} \myR \pi_* \omega_{Y}(- \lfloor \pi^*(K_X + \Delta) \rfloor) \to i_* \myH^{\dim Y-\dim X} \myR \pi_* \omega_{\pi^{-1}(U)}(- \lfloor \pi^*(K_X + \Delta) \rfloor)
\]
now gives the desired first map
\[
    \tr_{\pi} \: \myH^{-\dim X}\myR\pi_*\omega_Y(-\lfloor \pi^*(K_X+\Delta) \rfloor) \to \omega_X(-\lfloor K_X+\Delta\rfloor)
\]
by noting that $i_* \omega_U(-\lfloor (K_X + \Delta) \rfloor ) = \omega_X(- \lfloor (K_X + \Delta) \rfloor)$.  To see that $\tr_{\pi}$ is non-zero, localize to the generic point of $X$ (where $K_X+\Delta$ is trivial) and apply \autoref{thm.TraceNonzero}.

The construction of the second map $\tr_{\pi^{\mydot}}$ is similar, rather starting from \eqref{eq:tracedualizing} in place of \eqref{eq.TraceCanonical}.
Notice that we require that $\pi^*(K_X + \Delta)$ to be Cartier so that we have a means of interpreting $\omega_Y^{\mydot}(-\lfloor \pi^*(K_X+\Delta) \rfloor)$.  Since $\tr_{\pi^{\mydot}}$ agrees with $\tr_{\pi}$ generically by the proof of \autoref{thm.TraceNonzero}, this map is again non-zero.
\end{proof}

In the previous Proposition, we studied trace maps twisted by $\bQ$-divisors.  In the next Lemma, we study a special case of this situation which demonstrates that sometimes this trace map can be re-interpreted as generating a certain module of homomorphisms.% between $\O_X$-modules.

\begin{lemma}
  \label{lem:tracepremultdivisorcomputation}
Let $X$ be a normal integral $F$-finite scheme. Suppose that $\Delta$ is a $\bQ$-divisor such that $(p^{e} - 1)(K_{X} + \Delta) = \Div c$ for some $e > 0$ and $0 \neq c \in K(X)$.  If $\Phi_{X}^{e} \: F^{e} \omega_{X} \to \omega_{X}$ is the trace of the $e$-iterated Frobenius, then the homomorphism $\phi(\blank) = \Phi_{X}^{e}( F^{e}_{*} c \cdot \blank)$ generates $\Hom_{\O_{X}}(F^{e}_{*}\O_{X}(\lceil (p^{e}-1)\Delta \rceil), \O_{X})$ as an $F^{e}_{*}\O_{X}$-module.
\end{lemma}

\begin{proof}
  Essentially by construction (and the definition of $\omega_{X}^{\mydot}$), we have that $\Phi_{X}^{e}$ generates $\Hom_{\O_{X}}(F^{e}_{*}\omega_{X},\omega_{X})$ as an $F^{e}_{*}\O_{X}$-module.  %As in the proof of \autoref{lem:testidealmodulecorresp},
  Using the identification $\omega_{X} = \O_{X}(K_{X})$ we may consider $\Phi_{X}^e$ to generate
\[
\Hom_{\O_{X}}(F^{e}_{*}\O_{X}((1-p^{e})K_{X}),\O_{X}) = \Hom_{\O_{X}}(F^{e}_{*}\omega_{X},\omega_{X}) \, \, .
\]
But then multiplication by $c$ induces an isomorphism $\O_{X}((p^{e}-1)\Delta) \to[\cdot c] \O_{X}((1-p^{e})K_{X})$ (note $(p^{e}-1)\Delta$ is integral), so that $\Phi_{X}(F^{e}_{*}c \cdot \blank)$ generates $\Hom_{\O_{X}}(F^{e}_{*}\O_{X}(\lceil (p^{e}-1)\Delta \rceil), \O_{X})$ as an $F^{e}_{*}\O_{X}$-module as well.
\end{proof}

A main technique in this paper is the observation that the images of the various trace maps $\tr_\pi$ are preserved under the trace of the Frobenius. We will show this now for $\tr_\pi\: \myH^{\dim Y-\dim X} R\pi_*\omega_Y \to \omega_X$.  We will obtain a partial generalization involving $\bQ$-divisors within the proof of \autoref{prop.EasyContainmentViaGeneralMaps}.%; the more general case with divisors will follow in Proposition~\ref{prop.EasyContainmentViaAlteration} below.
\begin{proposition}\label{prop.EasyContainmentgeneral}
If $\pi \: Y \to X$ is a proper dominant map of integral schemes, the image of the trace map
\[
J_\pi \colonequals \tr_\pi(\myH^{\dim Y-\dim X}\myR\pi_*\omega_Y) \subseteq \omega_X
\]
satisfies $\Phi_{X}(F_*J_\pi) \subseteq J_\pi$.
\end{proposition}
\begin{proof}
Since Frobenius commutes with any map, we get the following diagram for which we consider the corresponding commutative diagram of structure sheaves
\[\xymatrixcolsep{3pc}\xymatrix{
    Y \ar[r]^\pi & X && \myR\pi_*\O_Y\ar[d]_{\myR\pi_*F} & \O_X \ar[l]_-\pi\ar[d]^F \\
    Y \ar[r]^\pi\ar[u]^F & X\ar[u]_F && F_*\myR\pi_*\O_Y & F_*\O_X \ar[l]_-{F_*\pi}
} \, \, .
\]
Applying duality now gives the following commutative diagram of trace maps
\begin{equation}\label{eq.traceFrob}
\xymatrixcolsep{5pc}
\xymatrix{
\myR\pi_*\dcx_Y\ar[r]^-{\tr_{\pi^{\mydot}}} & \dcx_X \\
F_*\myR\pi_*\dcx_Y\ar[u]^-{\myR\pi_*\tr_{F^{\mydot}}}\ar[r]^{F_*\Tr_{\pi^{\mydot}}} & F_*\dcx_X\ar[u]_{\tr_{F^{\mydot}}}
} \, \, .
\end{equation}
Taking the ($-\dim X$)-th cohomology and composing with the inclusion $\omega_Y[\dim Y] \to \dcx_Y$ on the left, we get a diagram
\begin{equation}
\label{eq:ultimatetracecommutes}
\xymatrix{
\myH^{\dim Y-\dim X}\myR\pi_*\omega_Y\ar[rr] && \myH^{-\dim X}\myR\pi_*\dcx_Y\ar[rr]^-{\tr_{\pi^{\mydot}}} && \omega_X \\
\myH^{\dim Y-\dim X}F_*\myR\pi_*\omega_Y\ar[u]^{\myH^{\dim Y - \dim X}\myR \pi_{*} \Phi_{Y}}\ar[rr] && \myH^{-\dim X}F_*\myR\pi_*\dcx_Y\ar[u] \ar[rr]^-{F_*\Tr_{\pi^{\mydot}}} && F_*\omega_X\ar[u]_{\Phi_{X}}
}
\end{equation}
where the left vertical map exists since $F$ is finite and hence $F_*$ is exact. The horizontal composition on the top is $\tr_\pi$ and the image $\tr_\pi(\myH^{\dim Y-\dim X}R\pi_*\omega_Y)$ is $J_\pi \subseteq \omega_X$. By the exactness of $F_*$, the horizontal composition on the bottom is $F_{*}\tr_{\pi}$ and we get that $F_*J_\pi$ is the image, and the result now follows.
\end{proof}

\subsection{Multiplier ideals and pseudo-rationality}\label{sec.MultiplierRational}
%\ok{\checkmark}{Fine with it}{Better?}
%In this section, whenever dealing with divisors, we assume that $X$ is a normal integral scheme.

\begin{definition} \cite{LipmanTeissierPseudoRational}\label{def.pseudo-rational}
We say that a reduced connected scheme $X$ is \emph{pseudo-rational} if
\begin{enumerate}
\item $X$ is Cohen-Macaulay, and
\item \label{def.pseudo-rational.b} $\pi_* \omega_Y = \omega_X$ for every proper birational map $\pi : Y \to X$.
\end{enumerate}
Furthermore, if there exists a resolution of singularities $\pi : Y \to X$, then it is sufficient to check \autoref{def.pseudo-rational.b} for this one map $\pi$.
\end{definition}

If $X$ is of characteristic zero, this coincides via Grauert-Riemenschneider vanishing \cite{GRVanishing} with the classical definition of rational singularities, meaning there exists a resolution of singularities $\pi : Y \to X$ such that $\O_X \cong \pi_* \O_{Y}$ and $\myH^i \myR \pi_* \O_{Y} = 0$ for all $i$.  In positive or mixed characteristic, it is a distinct notion.

\begin{remark}
It was remarked in \cite{GRVanishing} that if $\pi : Y \to X$ is a resolution of singularities in characteristic zero, then the subsheaf $\pi_* \omega_Y \subseteq \omega_X$ is independent of the choice of resolution of singularities.  This subsheaf should be viewed as an early version of the multiplier ideal. Compare with the definition of the multiplier module below and the parameter test submodule in \autoref{def.ParamTestSubmodule}.
\end{remark}

Going back to ideas in K. Smith's thesis and \cite{SmithTestIdeals}, the natural object to deal with rational singularities of pairs is the multiplier module, \textit{cf.} \cite{BlickleMultiplierIdealsAndModulesOnToric,SchwedeTakagiRationalPairs}.

\begin{definition}\label{def.multmod}
Given a pair $(X,\Gamma)$ with $\Gamma$ a $\bQ$-Cartier $\bQ$-divisor, then the $\emph{multiplier module}$ is defined as
\[
    \mJ(\omega_X; \Gamma) \colonequals \bigcap_{\pi\: Y \to X} \pi_*\omega_Y(\roundup{-\pi^*\Gamma})
\]
where $\pi$ ranges over all proper birational maps with normal $Y$.
\end{definition}

Note that from this definition it is not clear that $\mJ(\omega_X; \Gamma)$ is even quasi-coherent, as the infinite intersection of coherent subsheaves need not be quasi-coherent in general. However, if there is a theory of resolution of singularities available (for example over a field of characteristic zero \cite{HironakaResolution}, or in dimension $\leq 2$ \cite{LipmanResInDim2}), it is straightforward to check coherence by showing that the above intersection stabilizes.
 Recall that a \emph{log resolution} of the pair $(X, \Gamma)$ is a proper birational map $\pi : Y \to X$ with $Y$ regular and exceptional set $E$ of pure codimension one such that $\Supp(E) \cup \pi^{-1}(\Supp (\Gamma))$ is a simple normal crossings divisor.
Assuming every normal proper birational modification can be dominated by a log resolution, one can in fact show
\[
\mJ(\omega_X; \Gamma) =  \pi_* \omega_Y(\roundup{-\pi^*\Gamma} )
\]
for any single log resolution $\pi \: Y \to X$, which is in particular coherent.
Note that, for effective $\Gamma$ it is a subsheaf of $\omega_X$ via the natural inclusion $\pi_*\omega_Y \subseteq \omega_X$ as in \autoref{ex:tracebirational}.

Immediately from this definition it follows that $X$ is pseudo-rational if and only if $X$ is Cohen-Macaulay and $\mJ(\omega_X)\colonequals \mJ(\omega_X;0) = \omega_X$. Hence one defines:

\begin{definition}[\cite{SchwedeTakagiRationalPairs}]
A pair $(X,\Gamma)$ with $\Gamma \geq 0$ a $\bQ$-Cartier $\bQ$-divisor is called \emph{pseudo-rational} if $X$ is Cohen-Macaulay and $\mJ(\omega_X;\Gamma)=\omega_X$. Note that this implies that $\lfloor \Gamma \rfloor = 0$.
\end{definition}

The classical notion is of course that of multiplier ideals, which have been defined primarily in characteristic zero.  See \cite{LazarsfeldPositivity2} for a complete treatment in this setting. Historically, while multiplier ideals first appeared in more analytic contexts and were originally defined using integrability conditions, one facet of their pre-history was defined for any normal integral scheme -- J. Lipman's adjoint ideals \cite{LipmanAdjointsOfIdeals}.  The definition we give here (which makes sense in arbitrary characteristic) is a slight generalization of Lipman's definition to the modern setting of pairs.
\renewcommand{\J}{\mathcal{J}}

\begin{definition}[\cite{LazarsfeldPositivity2, LipmanAdjointsOfIdeals}]
\label{def:multiplierideals}
Given a log-$\bQ$-Gorenstein pair $(X, \Delta)$ then the \emph{multiplier ideal} is defined as
\[
\mJ(X; \Delta) \colonequals \bigcap_{\pi : Y \to X} \pi_* \O_Y(\roundup{K_Y - \pi^*(K_X + \Delta) })
\]
where $\pi$ ranges over all proper birational maps with normal $Y$ and, for each individual $\pi$, we have that $K_X$ and $K_Y$ agree wherever $\pi$ is an isomorphism.
\end{definition}
As above, in a non-local setting, for $\mJ(X;\Delta)$ to be quasi-coherent one needs a good theory of resolution of singularities. In this situation,
\[
\mJ(X; \Delta) =  \pi_* \O_Y(\roundup{K_Y - \pi^*(K_X + \Delta)} )
\]
for any log resolution $\pi \: Y \to X$ of the pair $(X,\Delta)$.  In general, $\J(X;\Delta)$ depends heavily on $\Delta$ and not simply the corresponding linear or $\Q$-linear equivalence class; a similar observation holds for the multiplier module as well.

If $(X, \Delta)$ is a pair, strictly speaking the object $\J(\omega_{X};K_{X} + \Delta)$ is ambiguous as $K_{X}$ is not uniquely determined (and represents a linear equivalence class of divisors).  Nonetheless, for each choice of $K_{X}$ we have that $\J(\omega_{X} ; K_{X} + \Delta) \subseteq \omega_{X}(-\lfloor K_{X} + \Delta\rfloor)$, and is thereby identified with a submodule of $\O_{X}(- \lfloor \Delta \rfloor)$ using Convention~\ref{rem.HowToEmbedOmegaXKX}.  This construction is in fact independent of the choice of $K_{X}$, and allows one to relate
multiplier ideals and multiplier modules in general.

\begin{lemma}
If $(X,\Delta)$ is a log-$\bQ$-Gorenstein pair, then $\mJ(\omega_X; (K_X + \Delta)) = \mJ(X; \Delta)$.
\end{lemma}
\begin{proof}
Suppose $\pi \: Y \to X$ is a proper birational map, $Y$ is normal, and $K_{Y}$ and $K_{X}$ agree wherever $\pi$ is an isomorphism.  Making full use of Convention~\ref{rem.HowToEmbedOmegaXKX}, we have
\[
\begin{array}{ccc}
  \omega_{X}( - \lfloor K_{X} + \Delta \rfloor ) & = & \O_{X}(-\lfloor \Delta \rfloor) \medskip \\
\rotatebox{90}{$\subseteq$} & & \rotatebox{90}{$\subseteq$} \\
\pi_{*} \omega_{Y}( - \lfloor \pi^{*}(K_{X} + \Delta) \rfloor) & = & \pi_{*} \O_{Y}(\lceil K_{Y} - \pi^{*}(K_{X}+ \Delta) \rceil)
\end{array}
\]
and the desired conclusion now follows immediately from the definitions.
\end{proof}

\begin{definition}
A log-$\bQ$-Gorenstein pair $(X, \Delta)$ with $\Delta \geq 0$ effective is called \emph{Kawamata log terminal} if $\mJ(X; \Delta) = \O_X$.
\end{definition}

\subsection{The parameter test submodule and $F$-rationality}%\ok{\checkmark}{No further comments}{--}
We now turn to the characteristic $p>0$ notion of $F$-rationality, extensively studied in \cite{FedderWatanabe} and \cite{SmithFRatImpliesRat}, which is central to our investigations.
\begin{definition}
Suppose that $X$ is reduced, connected, and $F$-finite (and satisfies Convention~\ref{rem.FUpperShrieckConvenient}).  We say that $X$ is \emph{$F$-rational} if
\begin{enumerate}
\item  $X$ is Cohen-Macaulay.
\item  There is no proper submodule $M \subseteq \omega_X$, non-zero on every irreducible component of $X$ where $\omega_{X}$ is non-zero, such that $\Phi_{X}(F_* M) \subseteq M$ where $\Phi_{X} \: F_* \omega_X \to \omega_X$ is the trace of Frobenius as in \autoref{ex:TraceForFrobenius}.
\end{enumerate}
\end{definition}
\begin{remark}
While the preceding characterization of $F$-rationality differs from
the definition used historically throughout the literature, it is nonetheless readily seen to be equivalent.  Indeed, when $X = \Spec R$ for a local ring $R$ with maximal ideal $\bm$, the characterization of the tight closure of zero in $H^{\dim R}_{\bm}(R)$  found in \cite{SmithFRatImpliesRat} implies that condition~(b) is Matlis dual to the statement
$0^*_{H^{\dim R}_{\bm}(R)} = 0$.
\end{remark}

\begin{definition}[\cite{SmithTestIdeals, BlickleMultiplierIdealsAndModulesOnToric, SchwedeTakagiRationalPairs}] \label{def.ParTestSub}
\label{def.ParamTestSubmodule}
Suppose that $X$ is reduced, connected, and \mbox{$F$-finite} (and satisfies Convention~\ref{rem.FUpperShrieckConvenient}).
The \emph{parameter test submodule} $\tau(\omega_{X})$ is the unique smallest subsheaf $M$ of $\omega_X$, non-zero on every irreducible component of $X$ where $\omega_X$ is nonzero, such that $\Phi_{X}(F_* M) \subseteq M$ where $\Phi_{X} : F_* \omega_{X} \to \omega_{X}$ is trace of Frobenius.

For $X$ normal and integral, the \emph{parameter test submodule} $\tau(\omega_{X};\Gamma)$ of a pair $(X, \Gamma)$ with $\Gamma \geq 0$ is the unique smallest non-zero subsheaf $M$ of $\omega_X$ such that $\phi(F^e_* M) \subseteq M$ for all \emph{local} sections $\phi \in \sHom_{\O_{X}}(F^e_* \omega_X( \lceil (p^e - 1)\Gamma \rceil), \omega_X)$ and all $e > 0$, noting that $\omega_X \subseteq \omega_X( \lceil (p^e - 1)\Gamma \rceil)$.  The further observation $\phi(F_{*}\omega_{X}(\lceil -\Gamma \rceil)) \subseteq \omega_{X}(\lceil -\Gamma \rceil)$ gives that $\tau(\omega_{X};\Gamma) \subseteq \omega_{X}(\lceil - \Gamma \rceil)$.
\end{definition}

Once again, the preceding definition is separate from but equivalent to that which is commonly used throughout the literature.  Moreover, standard arguments on the existence of test elements are required to show the (non-obvious) fact that $\tau(\omega_{X})$ and $\tau(\omega_{X};\Gamma)$ (as above) are well-defined. See, for example, \cite[Proposition 3.21]{SchwedeTestIdealsInNonQGor} (\cf \cite[Lemma 2.17]{SchwedeCentersOfFPurity}), or more generally \cite{BlickleBoeckleCartierModulesFiniteness, BlickleTestIdealsViaAlgebras}.

\begin{lemma}
Suppose that $X$ is reduced, connected, and \mbox{$F$-finite}.  Then $X$ is $F$-rational if and only if it is Cohen-Macaulay and $\tau(\omega_X) = \omega_X$. Furthermore, any $F$-rational scheme is normal.
\end{lemma}
\begin{proof}
The first statement is an immediate consequence of the definitions, so we need only show the second.
 Without loss of generality, we may assume that $X=\Spec R$ where $R$ is an $F$-rational local ring.  If $R^{\textnormal{N}}$ is the normalization of $R$, we will show the inclusion map $i : R \to R^{\textnormal{N}}$ is an isomorphism. As $R$ is already assumed Cohen-Macaulay it is $\textnormal{S}_2$, and so by Serre's criterion for normality we simply need to check that $R$ is regular in codimension $1$.  Thus, by localizing we may assume that $R$ is one dimensional (and thus so is $R^{\textnormal{N}}$, which now must be regular).
Consider the following commutative diagram of rings together with its corresponding Grothendieck-Serre dual (all rings in question are Cohen-Macaulay)
\[
\xymatrix{
R^{\textnormal{N}} \ar[r]^-{F} & F_*( R^{\textnormal{N}})  \\
R \ar[u]^{i} \ar[r]_-{F} & F_* R \ar[u]_{F_* i}
}  \qquad \qquad \xymatrix{
\omega_{R^{\textnormal{N}}}\ar[d]_{i^{\vee}} & \ar[l]_-{\Phi_{R^{\textnormal{N}}}} F_* \omega_{R^{\textnormal{N}}} \ar[d]^{F_* (i^{\vee})} \\
\omega_{R} & \ar[l]^{\Phi_{R}} F_* \omega_{R}
} \, \, .
\]
% \[
% \xymatrix{
% R^{\textnormal{N}} \ar[r]^-{F} & F_*( R^{\textnormal{N}}) && \omega_{R^{\textnormal{N}}}^{\mydot} \ar[d] & \ar[l] F_* \omega_{R^{\textnormal{N}}}^{\mydot} \ar[d] \\
% R \ar[u]^{i} \ar[r]_-{F} & F_* R \ar[u]_{F_* i} && \omega_{R}^{\mydot} & \ar[l] F_* \omega_{R}^{\mydot}
% } \, \, .
% \]
% All the rings in question are Cohen-Macaulay, so the second diagram is simply a diagram of canonical modules
% \[
% \xymatrix{
% \omega_{R^{\textnormal{N}}}\ar[d]_{i^{\vee}} & \ar[l] F_* \omega_{R^{\textnormal{N}}} \ar[d]^{F_* i^{\vee}} \\
% \omega_{R} & \ar[l] F_* \omega_{R}
% } \, \, .
% \]
As in \autoref{ex.traceForFinite}, $i^{\vee}$ is identified with the evaluation-at-1 map $\Hom_R(R^{\textnormal{N}}, \omega_{R}) \to \omega_{R}$.  Since $i$ is birational, it is easy to see $i^{\vee}$ is injective.  In particular,
$i^{\vee}$ is non-zero and thus also surjective by the definition of $F$-rationality. It follows that $i^{\vee}$ and hence $i$ are isomorphisms, whence $R$ is normal as desired.
% We thus simply need to show that $i^{\vee}$ is injective, because the Grothendieck dual of $i^{\vee} : \omega_{R^{\textnormal{N}}} \to \omega_R$ is $R \to R^{\textnormal{N}}$.
%%%%statement below doesn't make sense
% But $i^{\vee}$ is identified with the evaluation-at-1 map $\Hom_R(R^{\textnormal{N}}, R) \to R$ which is always injective for a birational map (its image is the conductor).
\end{proof}
In order to consider pairs $(X,\Delta)$ with $\Delta$ not necessarily effective, we need to recall the following lemma.

\begin{lemma}\cite[Proposition 7.10(3)]{SchwedeTakagiRationalPairs}\cite[Lemma 6.12]{SchwedeTuckerTestIdealFiniteMaps}
\label{lem.PullOutCartierForTestSubmod}
 Suppose that $(X, \Delta)$ is a pair with $\Delta \geq 0$, and that $D \geq 0$ an integral Cartier divisor on $X = \Spec R$, then $\tau(\omega_X; \Delta + D) = \tau(\omega_X; \Delta) \tensor \O_X(-D)$.
\end{lemma}

\begin{definition}
\label{def.NonEffectiveTestSubmodule}
Suppose that $(X, \Gamma)$ is a pair.  Fix a Cartier divisor $D$ on $X$ such that $\Gamma + D$ is effective (these always exist on affine charts).  Then the \emph{parameter test module of $(X, \Gamma)$} is defined as
\[
    \tau(\omega_X; \Gamma) := \tau(\omega_X; \Gamma + D) \tensor \O_X(D) \, \, .
\]
It is also easy to verify that this definition is independent of the choice of $D$, hence our local definition globalizes.  It is straightforward to check that $\tau(\omega_{X};\Gamma) \subseteq \omega_{X}(\lceil - \Gamma \rceil)$.
\end{definition}

In defining the test ideal of a pair below, we handle the non-effective case analogously.

\begin{definition}\label{def.TestIdeal}
Suppose that $X$ is reduced and $F$-finite. The \emph{test ideal} $\tau(X)$ is the unique smallest ideal $J$ of $\O_X$, non-zero on every irreducible component of $X$, such that $\phi(F^e_* J) \subseteq J$ for every \emph{local} section $\phi \in \sHom_{\O_X}(F^e_*\O_X,\O_X)$ and all $e >0$.

If $(X, \Delta)$ is a pair with $\Delta \geq 0$, the \emph{test ideal} $\tau(X;\Delta)$ is the unique smallest non-zero ideal $J$ of $\O_X$ such that $\phi(F^e_* J) \subseteq J$ for any \emph{local} section $\phi \in \sHom_{\O_X}(F^e_* \O_X( \lceil (p^e - 1) \Delta \rceil), \O_X)$ and all $e > 0$, noting that $J \subseteq \O_X \subseteq \O_X( \lceil (p^e - 1)\Delta \rceil)$.

As with the parameter test module, one has for any effective integral Cartier Divisor $D$ the equality $\tau(X;\Delta+D)=\tau(X;\Delta) \tensor \O_X(-D)$, which allows one to extend the definition to the non-effective case as above (see \cite{SchwedeTakagiRationalPairs} for further details).
Furthermore, the same subtle albeit well-known arguments are again required to show these ideals exist \cite[Proposition 3.21]{SchwedeTestIdealsInNonQGor} (see also \cite{SchwedeTuckerTestIdealSurvey}).
\end{definition}

\begin{remark}
As before, the preceding definition is non-standard; rather,
what we have just defined is an alternative yet equivalent characterization of the \emph{big} or \emph{non-finitistic} test ideal, commonly denoted in the literature by $\tau_b(X,\Delta)$ or $\tilde\tau(X,\Delta)$. However, in many situations (and conjecturally in general) the big test ideal agrees with the classically defined or \emph{finitistic} test ideal. Indeed, these two notions are known to coincide whenever $K_X + \Delta$ is $\bQ$-Cartier \cite{TakagiInterpretationOfMultiplierIdeals,BlickleSchwedeTakagiZhang} -- the only setting considered in this paper.  For this reason, as well as our belief that the big test ideal is the correct object of study in general,
we will drop the adjective \emph{big} from the terminology throughout.
\end{remark}

Strictly speaking the object $\tau(\omega_{X};K_{X} + \Delta)$ is ambiguous as $K_{X}$ is not uniquely determined.  Indeed supposing $\Delta \geq 0$, as seen from \autoref{lem.PullOutCartierForTestSubmod}, different choices of $K_{X}$ give rise to different submodules $\tau(\omega_{X};K_{X} + \Delta)$ of $\omega_{X}$.  However, we in fact have $\tau(\omega_{X} ; K_{X} + \Delta) \subseteq \omega_{X}(-\lfloor K_{X} + \Delta\rfloor)$, so that $\tau(\omega_{X};K_{X} + \Delta)$ is identified with a submodule of $\O_{X}(- \lfloor \Delta \rfloor)$ using Convention~\ref{rem.HowToEmbedOmegaXKX}.  As was the case with the multiplier module, this construction is independent of the choice of $K_{X}$.%, and allows one to relate test ideals and test modules in general.

\begin{lemma}
\label{lem:testidealmodulecorresp}
If $X$ is an $F$-finite and $(X, \Delta)$ is a pair, then
$\tau(\omega_X; K_X + \Delta) = \tau(X; \Delta)$.
\end{lemma}
\begin{proof}
Choose a rational section $s \in \omega_X$ determining a canonical divisor $K_X=K_{X,s} = \operatorname{div} s$ and the embedding
$\omega_{X} \subseteq \omega_X \tensor K(X) \to[s \mapsto 1] K(X)$.  We will write $\omega_{X} = \O_{X}(K_{X})$ for the duration of the proof without further remark.
%This gives the identification $\omega_X(-K_X) = \O_X \subseteq K(X)$ as in Convention \ref{rem.HowToEmbedOmegaXKX}.
Working locally, it is harmless to assume $K_{X}$ and $\Delta$ are both effective by Definition \ref{def.NonEffectiveTestSubmodule}, so that
$\tau(\omega_{X};K_{X} + \Delta) \subseteq \omega_{X}(-\lfloor K_{X} + \Delta \rfloor) = \O_{X}(-\lfloor \Delta \rfloor)$ gives in particular
$\tau(\omega_{X};K_{X} + \Delta) \subseteq \O_{X} \subseteq \omega_{X}$.

Next, observe for all $e > 0$ that
\[
\sHom_{\O_{X}}(F^{e}_{*}\omega_{X}(\lceil (p^{e}-1)(K_{X}+ \Delta )\rceil),\omega_{X} ) = \sHom_{\O_{X}}(F^{e}_{*}\O_{X}(\lceil (p^{e}-1)\Delta \rceil),\O_{X})
\]
in a very precise sense; they are equal after mapping to the corresponding stalks at the generic point of $X$, which are explicitly identified with one another.  Said another way, the image of $F^e_* \O_X( \lceil (p^e - 1) \Delta \rceil) \subseteq F^{e}_{*}\omega_{X}(\lceil (p^{e}-1)(K_{X}+ \Delta )\rceil)$ under every local homomorphism $\phi \in \sHom_{\O_{X}}(F^{e}_{*}\omega_{X}(\lceil (p^{e}-1)(K_{X}+ \Delta )\rceil),\omega_{X} )$ satisfies $\phi(F^e_* \O_X( \lceil (p^e - 1) \Delta \rceil)) \subseteq \O_{X}$, giving rise to a commutative diagram
\[
\begin{array}{ccccc}
F^{e}_{*}\omega_{X} &\subseteq& F^e_* \omega_X\left( \lceil (p^e - 1) (K_X + \Delta) \rceil \right) &\xymatrix{ \ar[r]^{\phi} &} & \omega_{X} \bigskip \\
\rotatebox{90}{$\subseteq$} && \rotatebox{90}{$\subseteq$} && \rotatebox{90}{$\subseteq$}  \\
F^{e}_{*} \O_{X} & \subseteq & F^e_* \O_X( \lceil (p^e - 1) \Delta \rceil)& \xymatrix{ \ar[r]^{\phi} &} & \O_{X}
\end{array}
\]
and uniquely accounting for every local homomorphism in $\sHom_{\O_{X}}(F^e_* \O_X( \lceil (p^e - 1) \Delta \rceil), \O_{X})$.

% \bigskip
% $\O_{X} \subseteq \omega_{X}$ and
% \[
% \tau(\omega_{X};K_{X} + \Delta) \subseteq \omega_{X}(-\lfloor K_{X} + \Delta \rfloor) = \O_{X}(-\lfloor \Delta \rfloor) \subseteq \O_{X}
% \]

% \medskip
% $\tau(\omega_X, K_X + \Delta) \subseteq \O_X \subseteq \O_{X}(K_{X}) = \omega_{X}$.

% Now, any map $\phi : F^e_* \omega_X\left( \lceil (p^e - 1) (K_X + \Delta) \rceil \right) \to \omega_X$ can be viewed as a map
% \begin{equation}
% \label{eq:phiMapOnKX}
%  \phi : F^e_* \O_X( K_X + \lceil (p^e - 1) (K_X + \Delta) \rceil) \to \O_X(K_X)
% \end{equation}
% and so can be restricted to the following map, which we abusively also call $\phi$:
% \begin{equation}
% \label{eq:phiMapOnOX}
%  \phi : F^e_* \O_X( K_X - p^e K_X + \lceil (p^e - 1) (K_X + \Delta) \rceil) = F^e_* \O_X( \lceil (p^e - 1) \Delta \rceil) \to \O_X.
% \end{equation}
% It is thus clear that $\O_X$ is $\phi$-stable, so $\tau(\omega_X, (K_X + \Delta))$ is already an ideal inside $\O_X$.  On the other hand, every map coming from \eqref{eq:phiMapOnOX} extends (uniquely) to a map of the form \eqref{eq:phiMapOnKX}.  Therefore, when computing the smallest $\O_X$-submodule of $\omega_X$ (respectively $\O_X$), we can use either \eqref{eq:phiMapOnKX} or \eqref{eq:phiMapOnOX}.  The statement follows.

The desired conclusion is now an exercise in manipulating definitions.  As $ \tau(X; \Delta) \subseteq \omega_{X}$ is preserved under the local homomorphisms in $\sHom_{\O_{X}}(F^{e}_{*}\omega_{X}(\lceil (p^{e}-1)(K_{X}+ \Delta )\rceil),\omega_{X} )$, we must have $\tau(\omega_{X};K_{X}+ \Delta) \subseteq \tau(X;\Delta)$ by the definition of $\tau(\omega_{X};K_{X}+ \Delta)$.  Conversely, as $\tau(\omega_{X};K_{X}+ \Delta) \subseteq \O_{X}$ is preserved under $\sHom_{\O_{X}}(F^{e}_{*}\O_{X}(\lceil (p^{e}-1)\Delta \rceil),\O_{X})$, we  must have $\tau(X;\Delta) \subseteq \tau(\omega_{X};K_{X}+ \Delta)$ and the statement follows.
\end{proof}

\subsection{Transformation behavior of test ideals and multiplier ideals}%\ok{--}{--}{--}

One of the contributions of this paper is a further clarification of the transformation behavior of test and multiplier ideals. Let us summarize what is known so far for alterations, which can always be viewed as compositions of proper birational maps and finite dominant maps.

Let us first consider the classical (and transparent) case of the multiplier ideal in characteristic zero \cite[9.5.42]{LazarsfeldPositivity2}.  Essentially by definition,  the multiplier ideal of a log-\mbox{$\Q$-Gorenstein} pair $(X,\Delta_{X})$ is well-behaved under a proper birational morphism $\pi \: Y \to X$ and satisfies
\begin{equation}
\label{eq:multbirat}
   \pi_*\mJ(Y,\Delta_Y)=\mJ(X,\Delta_X) \quad \mbox{with} \quad \Delta_Y \colonequals \pi^*(K_X+\Delta_X)-K_Y
\end{equation}
where we have arranged that $\pi_*K_Y=K_X$ as usual.  If rather $\pi \: Y \to X$ is a finite dominant map, one sets $K_{Y} = \pi^{*}K_{X} + \Ram_{\pi}$ where $\Ram_{\pi}$ is the ramification divisor, and has the transformation rule
\begin{equation}
\label{eq:multfiniteint}
    \pi_*\mJ(Y,\Delta_Y) \cap K(X) = \mJ(X,\Delta_{X}) \quad \mbox{with} \quad \Delta_Y \colonequals \pi^*(K_X+\Delta_X)-K_Y = \pi^{*}\Delta_X- \Ram_{\pi} \, \, .
\end{equation}
In \autoref{sec.TransformationRulesForMultiplier}, we further generalize this rule to incorporate the trace map \autoref{ex.traceForFinite}
\begin{equation}
\label{eq:multfinitetrace}
\Tr_{\pi}(\pi_*\mJ(Y,\Delta_Y)) = \J(X,\Delta_{X})
\end{equation}
in the process of showing our main theorem in characteristic zero.

In characteristic $p>0$, the transformation rule \eqref{eq:multbirat} for the multiplier ideal under proper birational maps once again follows immediately from the definition.  However, the behavior of the multiplier ideal for finite maps is more complicated and not fully understood in general -- even for finite separable maps.  In \autoref{sec.TransformationRulesForMultiplier}, we will show that both \eqref{eq:multfiniteint} and \eqref{eq:multfinitetrace} hold for separable finite maps of degree prime to $p$ (and more generally when $\tr_{\pi}(\pi_{*}\O_{Y}) = \O_{X}$).  However,  Examples 3.10, 6.13, and 7.12 in \cite{SchwedeTuckerTestIdealFiniteMaps} together show that neither formula is valid for arbitrary separable finite maps in general.

In contrast, the last two authors in \cite{SchwedeTuckerTestIdealFiniteMaps} have completely described the behavior of the test ideal of a pair $(X, \Delta_{X})$ under arbitrary finite dominant maps $\pi \: Y \to X$.  In the separable case, one again has
\begin{equation*}
\Tr_{\pi}(\pi_*\tau(Y,\Delta_Y)) = \tau(X,\Delta_{X}) \quad \mbox{with} \quad \Delta_Y \colonequals \pi^*(K_X+\Delta_X)-K_Y = \pi^{*}\Delta_X- \Ram_{\pi} \, \, .
\end{equation*}
More generally, when $\pi$ is not necessarily separable, a similar description holds after reinterpretation of the ramification divisor (via the Grothendieck trace).  However, a formula as simple as \eqref{eq:multbirat} relating test ideals under a birational map cannot hold.  Indeed, if $\pi \: Y \to X$ is a log resolution of $(X, \Delta_{X})$,
then the multiplier and test ideal of $(Y, \Delta_{Y})$ will agree while those of $(X,\Delta_X)$ may not (\cf \cite[Theorems 2.13, 3.2]{TakagiInterpretationOfMultiplierIdeals}).  Nonetheless in \autoref{thm.TransformationOfTestIdealsUnderDominantMaps} we will give a transformation rule -- albeit far more complex -- for the test ideal under alterations in general, so in particular for birational morphisms.

In summary we observe that the test ideal behaves well under finite maps (and may be computed using either finite maps or alterations by \autoref{thm.MainThm}), whereas its behavior under birational maps is much more subtle. On the other hand, the multiplier ideal in characteristic zero behaves well under finite and birational maps, however finite maps will not suffice for its computation. The multiplier ideal in positive characteristic is still well behaved under birational transformations, but its behavior under finite maps is more subtle.

\section{$F$-rationality via alterations}%\ok{\checkmark}{--}{--}

In this section, we will characterize $F$-rationality and, more generally, the parameter test submodule in terms of alterations.  This is -- at the same time -- a special case of our Main Theorem as well as a key ingredient in its proof. The full proof of our Main Theorem in positive characteristic consists of a reduction to the cases treated here and will follow in the next section.

The crux of the argument to follow is based on a version of the \emph{equational lemma} of \cite{HochsterHunekeInfiniteIntegralExtensionsAndBigCM} in the form that is found in \cite{HunekeLyubeznikAbsoluteIntegralClosure}. In fact, we require a variant with an even stronger conclusion; namely, that the guaranteed finite extension may be assumed \emph{separable}. This generalization follows from a recent result of A. Sannai and A. Singh \cite{sannai_galois_2011}.

\begin{lemma}[equational lemma] \label{thmEquationalLemma}
  Consider a domain $R$ with characteristic $p > 0$.  Let $K$ be the fraction field of $R$, $\bar{K}$ an algebraic closure of $K$, and $I$ an ideal of $R$. Suppose $i \geq 0$ and $\alpha \in H^i_I(R)$ is such that $\alpha, \alpha^p, \alpha^{p^2}, \dots$ belong to a finitely generated $R$-submodule of $H^i_I(R)$.  Then there exists a separable \mbox{$R$-subalgebra} $R'$ of $\bar{K}$ that is a finite $R$-module such that the induced map $H^i_I(R) \to H^i_I(R')$ sends $\alpha$ to zero.
\end{lemma}

\begin{proof}
The statement of the equational lemma in \cite[Lemma 2.2]{HunekeLyubeznikAbsoluteIntegralClosure} is the same as above without the desired separability. Applying this weaker version we therefore have a finite extension $R \subseteq R'$ such that $H^i_I(R) \to H^i_I(R')$ maps $\alpha$ to zero. Let $R^s$ be the separable closure of $R$ in $R'$, that is all elements of $R'$ that are separable over $R$. The extension $R^s \subseteq R'$ is then purely inseparable, \textit{i.e.} some power of the Frobenius has the property that $F^i(R') \subseteq R^s$. Applying the functor $H^i_I(\blank)$ yields the diagram
\[
\xymatrix{
    H^i_I(R) \ar[r] &H^i_I(R^s) \ar[r]\ar[d]^{F^i} &H^i_I(R') \ar[ld]^{F^i} \\
                    &H^i_I(R^s)
}
\]
Denoting the image of $\alpha$ in $H^i_I(R^s)$ by $a$, this element is mapped to zero $H^i_I(R')$, hence, under the Frobenius $F^i$ it is mapped also to zero in $H^i_I(R^s)$. But this shows that $F^i(a)=0$ in $H^i_I(R^s)$ by the above diagram. Now \cite[Theorem 1.3(1)]{sannai_galois_2011} states that there is a module finite \emph{separable} (even Galois with solvable Galois group) extension $R^s \subseteq R^{\prime\prime}$ such that $a$ is mapped to zero in $H^i_I(R^{\prime\prime})$. But this means that the finite separable extension $R \subseteq R^{\prime\prime}$ is as required: the image of $\alpha \in H^i_I(R)$ in $H^i_I(R^{\prime\prime})$ is zero.
\end{proof}

The main result of this section, immediately below, is a straightforward application of the method of C. Huneke and G. Lyubeznik in \cite{HunekeLyubeznikAbsoluteIntegralClosure}.  Since making this observation, we have been informed that a Matlis dual version of the theorem below has also been obtained by M. Hochster and Y. Yao (in a non-public preprint \cite{HochsterYaoUnpublished}).

\begin{theorem}\label{thm.CharOfParTestSubmodule}
Suppose $X$ is an integral $F$-finite scheme.
\begin{enumerate}
\item For all proper dominant maps $\pi\: Y \to X$ with $Y$ integral, the image of the trace map \eqref{eq.TraceCanonical} contains the parameter test module, \textit{i.e.}
\[
\tau(\omega_X) \subseteq \Image(\myH^{\dim Y - \dim X}\myR \pi_*\omega_Y  \to[\tr_{\pi}] \omega_X) \, \, .
\]
 \label{thm.CharOfParTestSubmodule.a}
\item
There exists a finite separable map $\pi \: Y \to X$ with $Y$ integral such that the image of the trace map equals the parameter test module, \textit{i.e.}
\[
\tau(\omega_X) = \Image(\myH^{\dim Y - \dim X}\myR\pi_*\omega_Y  \to[\tr_{\pi}] \omega_X) \, \, .
\]
%There exists a finite separable ring extension $R \subseteq S$ such that, for the induced morphism $\pi \: Y = \Spec S \to R$ of affine schemes, the image of the trace map $\tr_{\pi}(\pi_{*} \omega_{Y}) = \tau(\omega_{X}) \subseteq \omega_{X}$ equals the parameter test module.
\label{thm.CharOfParTestSubmodule.b}
\end{enumerate}
\end{theorem}

\begin{remark}
In the above result, it is possible to work with equidimensional reduced (rather than integral) schemes of finite type over a field at the expense of  removing ``separable'' from the conclusion in (b).
\end{remark}

An immediate Corollary of this statement is a characterization of the test submodule.  Recall that an alteration is a generically finite proper dominant morphism of integral schemes.
\begin{corollary}\label{thm.CharOfParTestSubmoduleCorollary}
Suppose $X$ is an integral $F$-finite scheme.
 Then
\[
\tau(\omega_X) = \bigcap_{\pi\: Y \to X} \Image(\myH^{\dim Y - \dim X}\myR\pi_*\omega_Y \to[\tr_{\pi}] \omega_X).
\]
where $\pi$ ranges over all of the maps from an integral scheme $Y$ to $X$ that are either:
\begin{itemize}
\item (separable) finite dominant maps, or
\item (separable) alterations, or
\item  (separable) proper dominant maps.
\end{itemize}
Additionally, if $X$ is a variety over a perfect field, then we may allow $\pi$ to range over
\begin{itemize}
 \item all regular separable alterations to $X$, or
 \item all separable proper dominant maps with $Y$ regular.
\end{itemize}
Furthermore, in this case, there always exists a separable regular alteration $\pi\: Y \to X$ such that $\Image(\pi_* \omega_Y \to[\tr_{\pi}] \omega_X)$ is equal to the parameter test submodule of $X$.
\end{corollary}

\begin{proof}
This follows immediately from \autoref{thm.CharOfParTestSubmodule} and from the existence of regular alterations \cite{de_jong_smoothness_1996}.
\end{proof}

\begin{proof}[Proof of \autoref{thm.CharOfParTestSubmodule} (a)]
The statement follows immediately from Propositions \ref{prop.EasyContainmentgeneral} and \ref{thm.TraceNonzero} as well as the definition of $\tau(\omega_X)$.
\end{proof}

The proof of \autoref{thm.CharOfParTestSubmodule.b} follows closely the strategy of \cite{HunekeLyubeznikAbsoluteIntegralClosure}; note that a local version of the statement is also related to the result of K. Smith that ``plus closure equals tight closure for parameter ideals'' \cite{SmithParameterIdeals}. The proof comes down to Noetherian induction once we establish the following lemma.

\begin{lemma}
\label{lem.TechnicalCohomologyKillingLemma}
Let $R \subseteq S$ be a module finite inclusion of domains and denote the image of the trace map by $J_{S} = \Image (\omega_S \to \omega_R)$. If $\tau(\omega_R) \subsetneq J_S$, then there is a separable finite extension of domains $S \subseteq S'$ such that the support of $J_{S'}/\tau(\omega_R)$ is strictly contained in the support of $J_S/\tau(\omega_R)$, where $J_{S'}=\Image(\omega_{S'} \to \omega_R)$.
\end{lemma}
\begin{proof}
Choose $\eta \in \Spec R$ to be the generic point of a component of the support of $J_S/\tau(\omega_R)$ and set $d= \dim R_\eta$. By construction $(J_S/\tau(\omega_R))_\eta=J_{S_\eta}/\tau(\omega_{R_\eta})$ has finite length, and hence so also must its Matlis dual
\[
    (J_S/\tau(\omega_R))_\eta^\vee=\Hom(J_{S_\eta}/\tau(\omega_{R_\eta}),E(R_\eta/\eta))\, .
\]
By Matlis duality applied to the sequence $\omega_{S_\eta} \onto (J_{S_\eta}/\tau(\omega_{R_\eta})) \into \omega_{R_\eta}/\tau(\omega_{R_\eta})$ one observes that $(J_S/\tau(\omega_R))_\eta^\vee \subseteq H^d_\eta(S_\eta)$ is identified with the image of the tight closure of zero $0^*_{H^d_\eta(R_\eta)}=(\omega_{R_\eta}/\tau(\omega_{R_\eta}))^\vee$ in $H^d_\eta(S_\eta)$. By
\autoref{prop.EasyContainmentgeneral}  we have for $\Phi_R : F_* \omega_R \to \omega_R$ that $\Phi_R(J_S) \subseteq J_S$. Hence the dual of $J_S/\tau(\omega_R)$ is stable under the action of the Frobenius on $H^d_\eta(S_{\eta})$. \ie $F((J_S/\tau(\omega_R))_\eta^\vee )\subseteq (J_S/\tau(\omega_R))_\eta^\vee$  (phrased differently: the tight closure of zero is Frobenius stable, hence so is its image). This implies that for any element $\alpha \in (J_S/\tau(\omega_R))_\eta^\vee$ the powers $\alpha, \alpha^p, \alpha^{p^2},\ldots$ must also lie in the finite length $(J_{S_\eta}/\tau(\omega_{R_\eta}))^\vee$. Applying \autoref{thmEquationalLemma} to $\alpha \in (J_{S_\eta}/\tau(\omega_{R_\eta}))^\vee$ repeatedly (\textit{e.g.} for a finite set of generators) we obtain a separable integral extension $S_\eta \subseteq T$ such that $H^d_\eta(S_\eta) \to H^d_\eta(T)$ maps $(J_{S_\eta}/\tau(\omega_{R_\eta}))^\vee$ to zero. By taking $S'$ to be the normalization of $S$ in the total field of fractions of $T$ we see that $T=S'_\eta$ and that the finite extension $R \subseteq S'$ is separable. Translating this back via Matlis duality this means that the map $\omega_{S'_\eta} \to \omega_{R_\eta}$ sends $J_{S_\eta'}$ into $\tau(\omega_{R_\eta})$. In particular $\eta \not\in \Supp (J_{S'}/\tau(\omega_R))$.
\end{proof}

\begin{proof}[Proof of \autoref{thm.CharOfParTestSubmodule} (b)]
Without loss of generality, we may assume that $X = \Spec R$ is affine.
Starting with the identity $R=S_0$ we successively produce, using \autoref{lem.TechnicalCohomologyKillingLemma}, a sequence of separable finite extensions $R=S_0 \subseteq S_1 \subseteq S_2 \subseteq S_3 \ldots$ such that the support of $J_{S_{i+1}}/\tau(\omega_R)$ is strictly smaller than the support of $J_{S_i}/\tau(\omega_R)$ until $J_{S_{i}}=\tau(\omega_R)$ by Noetherian induction.
\end{proof}

The following important corollary (which can also be proven directly from the equational lemma without reference to the above results) should be viewed in the context of the definition of pseudo-rationality (see \autoref{sec.MultiplierRational}), as well as the Kempf-criterion for rational singularity \cite[p. 50]{KempfToroidalEmbeddings} in characteristic zero.

\begin{corollary}
\label{cor.CharOfFRational}
For an $F$-finite Cohen-Macaulay domain $R$ the following are equivalent.
\begin{enumerate}
    \item $R$ is $F$-rational, \ie there is no non-trivial submodule $M \subseteq \omega_R$ which is stable under $\Phi_R: F_*\omega_R \to \omega_R$.
    \item For all finite extensions $R \to S$ (which may be taken to be separable if desired) the induced map $\omega_S \to \omega_R$ is surjective.
    \item For all alterations $\pi:Y \to X = \Spec R$ ($\pi$ may be taken to be separable and or regular if $R$ is of finite type over a perfect field) the induced map $\pi_* \omega_Y \to \omega_X$ is surjective.
\end{enumerate}
\end{corollary}

In fact, utilizing local duality and \cite{HunekeLyubeznikAbsoluteIntegralClosure} once again to obtain a further finite cover which annihilates the local cohomology modules below the dimension, we obtain the following characterization of \mbox{$F$-rationality} without the Cohen-Macaulay hypothesis.

\begin{corollary}
For an $F$-finite domain $R$, the following are equivalent:
\begin{enumerate}
\item  $R$ is $F$-rational.
\item For each integer $i \in \bZ$, and every (separable) finite extension of rings $R \to S$, the induced map $\myH^i\omega_S^{\mydot} \to \myH^i\omega_R^{\mydot}$ is surjective.
\item For each integer $ i \in \bZ$ and every (separable regular, if $R$ is of finite type over a perfect field) alteration $\pi : Y \to X = \Spec R$ , the induced map $\myH^i \myR \pi_* \omega_Y^{\mydot} \to \myH^i \omega_X^{\mydot}$ is surjective.
\end{enumerate}
\end{corollary}

\begin{remark}
The main result of the paper \cite{HunekeLyubeznikAbsoluteIntegralClosure} which inspired our proof is that for a local ring $(R,\bm)$ of dimension $d$ that is a homomorphic image of a Gorenstein ring, there is a module finite extension $R \subseteq S$ such that the induced map $H^i_\bm(R) \to H^i_\bm(S)$ is zero for $i < d$. A local dual statement to this is that the induced map on dualizing complexes $\omega_S^\mydot \to \omega_R^\mydot$ is zero on cohomology $\myH^i(\omega_S^\mydot) \to \myH^i(\omega_R^\mydot)$ for $i\neq - \dim R$. What we accomplish here is that we clarify the case $i=-\dim R$. With $d=\dim R$ we just showed that one can also achieve that the image $\myH^{-d}(\omega_S^\mydot) \to \myH^{-d}(\omega_R^\mydot) \cong \omega_R$ is the parameter test submodule $\tau(\omega_R)$.  Of course, the dual statement thereof is: the tight closure of zero $0^*_{H^d_\bm(R)}$ is mapped to zero under the map $H^d_\bm(R) \to H^d_\bm(S)$.  It is exactly this statement, and further generalizations, which are first and independently by M. Hochster and Y. Yao in \cite{HochsterYaoUnpublished}.
\end{remark}

\section{Test ideals via alterations}%\ok{--}{--}{--}

In this section, we explore the behavior of test ideals under proper dominant maps and prove our main theorem in characteristic $p > 0$ in full generality.  %We begin with some further investigations of the trace map and its images.
First we show that various images are compatible with the $\Phi$ from Example \ref{ex:TraceForFrobenius}, and so they contain the test ideal.

%We now state the next result in two ways, relating.  First we state it merely for alterations.

\begin{proposition}
\label{prop.EasyContainmentViaAlteration}
Suppose that $f : Y \to X$ is a proper dominant generically finite map of normal $F$-finite varieties and that $(X, \Delta)$ is a log-$\bQ$-Gorenstein pair. Then the test ideal is contained in the image of the trace map, \textit{i.e.}
\[
\tau(X; \Delta) \subseteq \Image\left( f_* \omega_Y( -\lfloor f^*(K_X + \Delta) \rfloor) \to[\tr_{f}] \omega_X(-\lfloor K_X + \Delta \rfloor) \subseteq K(X) \right)
\]
where $\tr_{f}$ is the map induced by trace as in \autoref{prop.TraceInSomeGenerality}.
\end{proposition}
\begin{proof}
This follows easily by Stein factorization.  Factor $f$ as $Y \xrightarrow{g} Z \xrightarrow{h} X$ where $g$ is birational and $h$ is finite.
Then it follows from \cite[Theorem 6.25]{SchwedeTuckerTestIdealFiniteMaps} (\cf \cite[Lemma 4.4(a)]{SchwedeTuckerNonprincipal}) that
\[
\tr_h(h_* \tau(Z; -K_Z + h^*(K_X + \Delta))) = \tau(X; \Delta).
\]
On the other hand, it follows from the argument that the test ideal is contained in the multiplier ideal (since the test ideal is the unique smallest ideal satisfying a certain property), that
\[
\tr_g(g_* \omega_Y(-\lfloor f^*(K_X + \Delta) \rfloor)) \supseteq \tau(Z; -K_Z + h^*(K_X + \Delta)).
\]
See \cite{TakagiInterpretationOfMultiplierIdeals} or see \cite[Theorem 4.17]{SchwedeTuckerTestIdealSurvey} for a sketch of a simpler version of this argument (or see immediately below for a generalization).

This completes the proof since $\tr_h \circ h_* \tr_g = \tr_f$.
\end{proof}

We also prove a more general statement whose proof also partially generalizes \autoref{prop.EasyContainmentgeneral}, and which uses heavily the material from the preliminary section on Duality, \autoref{sec.Duality}.  The reader who has so far avoided that section, can skip \autoref{prop.EasyContainmentViaGeneralMaps} and rely on \autoref{prop.EasyContainmentViaAlteration} instead.

\begin{proposition}
\label{prop.EasyContainmentViaGeneralMaps}
Suppose that $f : Y \to X$ is a proper dominant map of normal integral $F$-finite schemes and that $(X, \Delta)$ is a log-$\bQ$-Gorenstein pair. Then the test ideal is contained in the image of the trace map, \textit{i.e.}
\[
\tau(X; \Delta) \subseteq \Image\left(\myH^{\dim Y - \dim X} \myR f_* \omega_Y( -\lfloor f^*(K_X + \Delta) \rfloor) \to[\tr_{f}] \omega_X(-\lfloor K_X + \Delta \rfloor) \subseteq K(X) \right)
\]
where $\tr_{f}$ is the map induced by trace as in \autoref{prop.TraceInGenerality}.  Similarly, if additionally $f^*(K_X + \Delta)$ is a Cartier divisor, then
\[
\tau(X;\Delta) \subseteq \Image\left(\myH^{- \dim X} \myR f_* \omega_Y^{\mydot}(f^*(K_X + \Delta)) \to[\tr_{f^{\mydot}}] \omega_X(-\lfloor K_X + \Delta \rfloor) \subseteq K(X) \right)
\]
where again $\tr_{f^{\mydot}}$ is the map induced by trace as in \autoref{prop.TraceInGenerality}.
\end{proposition}

\begin{proof}
The statement is local (if it fails to hold, then it fails to hold locally), so we assume that $X$ is the spectrum of a local ring.
Without loss of generality, as in \autoref{def.TestIdeal}, we may assume that $\Delta \geq 0$. Let $n = \dim Y - \dim X$.

Fix an $\O_{X}$-linear map $\phi \: F^e_* \O_X \to \O_X$ such that $\Delta_{\phi} \geq \Delta$, where $\Delta_{\phi}$ is defined as in \cite[Theorem 3.11, 3.13]{SchwedeFAdjunction} or \cite[Subsection 4.4]{SchwedeTuckerTestIdealSurvey}.  Recall that $(p^e-1)(K_X + \Delta_{\phi}) \sim 0$, so in particular we may write $(p^{e} - 1)(K_{X}+ \Delta) = \Div c$ for some $c \in K(X)$.  For the first statement, it is sufficient to show that
\[
\phi\big( F^{e}_{*}\tr_{f}\left(\myH^n \myR f_* \omega_Y( -\lfloor f^*(K_X + \Delta) \rfloor) \right) \big) \subseteq \tr_{f}\big(\myH^{n} \myR f_* \omega_Y( -\lfloor f^*(K_X + \Delta) \rfloor) \big)
\]
and by \autoref{lem:tracepremultdivisorcomputation}, we may assume $\phi(\blank) = \Phi_{X}^{e}(F^{e}_{*}c \cdot \blank)$.  Now we have
\[
\phi(F^{e}_{*}\tr_{f}(\blank)) = \Phi_{X}^{e}\left( (F^{e}_{*}c) \cdot F^{e}_{*}\tr_{f}(\blank)\right) = \Phi_{X}^{e}(F^{e}_{*}\tr_{f}(c \cdot \blank)) = \tr_{f}(\myH^{n}\myR f_{*}\Phi_{Y}^{e}(F^e_* c \cdot \blank))
\]
where we have used that $\tr_{f}$ is $\O_{X}$-linear, and that $\Phi_{X}^{e}(F^{e}_{*}\tr_{f}(\blank)) = \tr_{f}(\myH^{n}\myR f_{*}\Phi_{Y}^{e}(F^e_* \blank))$ as shown in \eqref{eq:ultimatetracecommutes}.  Thus, it is enough to show
\[
\Phi_{Y}^{e}\big(F^e_* ( c \cdot \omega_{Y}( - \lfloor f^{*}(K_{X} + \Delta)\rfloor) )\big)\subseteq \omega_{Y}( - \lfloor f^{*}(K_{X} + \Delta)\rfloor) \, \, .
\]
Since
\[
- \lfloor f^{*}(K_{X} + \Delta) \rfloor + (1-p^{e})f^{*}(K_{X} + \Delta_{\phi}) = - \lfloor f^{*}( p^{e}K_{X} + (p^{e} - 1) \Delta_{\phi} + \Delta )\rfloor
\]
and $\Delta_{\phi} \geq \Delta$, we have
\[
c \cdot \omega_{Y}(- \lfloor f^{*}(K_{X}+ \Delta)\rfloor) \subseteq \omega_{Y}(-\lfloor p^{e}f^{*}(K_{X} + \Delta) \rfloor) = \omega_{Y}(-\lfloor (F^{e})^{*}f^{*}(K_{X}+\Delta)\rfloor) \, \, .
\]
But then, according to \autoref{prop.TraceInGenerality}, we have
\[
\Phi_{Y}^{e}\left( F^{e}_{*}\omega_{Y}(-\lfloor (F^{e})^{*}f^{*}(K_{X}+\Delta)\rfloor) \right) \subseteq \omega_{Y}(- \lfloor  f^{*}(K_{X}+\Delta)\rfloor)
\]
which completes the proof of the first statement. For the second statement, simply notice that $\Image(\Tr_{f^{\mydot}}) \supseteq \Image(\Tr_{f})$.
\end{proof}

\begin{lemma}
\label{lem.ConstructiveDescriptionOfTestIdeals}
Suppose that $(X = \Spec R, \Gamma)$ is a pair such that $\Gamma = t \Div(g)$ for some $g \in R$ and $t \in \bQ_{\geq 0}$.  Fix $c \in R$ such that $R_c$ is regular and that $\Supp(\Gamma) = V(g) \subseteq V(c)$. Then there exists a power of $c^N$ of $c$ such that
\[
\tau(\omega_X; \Gamma) = \sum_{e \geq 0} \Phi^e( F^e_* c^N g^{\lceil t(p^e - 1) \rceil} \omega_R).
\]
where $\Phi_R^e : F^e_* \omega_R \to \omega_R$ denotes the trace of the $e$-iterated Frobenius, see \autoref{sec.Duality}.
\end{lemma}
\begin{proof}
By the usual theory of test elements (\textit{cf.} \cite[proof of Theorem 3.18]{SchwedeTestIdealsInNonQGor}), we have for some power $c^n$ of $c$ that
\[
\tau(\omega_X; \Gamma) = \sum_{e \geq 0} \sum_{\phi} \phi(F^e_* c^n \omega_R)
\]
where the inner sum ranges over all $\phi \in \Hom_{R}( F^e_* \omega_R(\lceil (p^e - 1) \Gamma \rceil) , \omega_R)$.   Furthermore, note it is harmless to replace $n$ by any larger integer $n + k$.  The reason the statement does not follow immediately is that, as $\Div(g)$ may not be reduced, we may have $\lceil t(p^e -1 ) \rceil \Div(g) \geq \lceil t(p^e - 1) \Div(g) \rceil$.  Choose $k$ such that $\Div(c^k) + (p^e - 1) \Gamma \geq \Div(g^{\lceil t(p^e - 1) \rceil})$ for all $e \geq 0$, and set $N = n + k$.

Now, the map $\psi(\blank) = \Phi^e_{R}(g^{\lceil t(p^e - 1) \rceil} \cdot \blank) \in \Hom_R(F^e_* \omega_{R}(\lceil (p^e - 1)\Gamma \rceil), \omega_{R})$ appears in the sum above.  This implies the containment $\supseteq$ for our desired equality.  Furthermore,
for any $\phi \in \Hom_R(F^e_* \omega_{R}(\lceil (p^e - 1)\Gamma \rceil), \omega_{R})$, it is clear then that $\phi(F^e_* c^{N} \omega_R) \subseteq F^e_* \Phi^e(F^e_* c^n g^{\lceil (p^e - 1) \rceil} \omega_R)$.  Thus, we have the containment $\subseteq$ as well.
\end{proof}

We now describe the transformation rule for the parameter test module under finite maps.

\begin{proposition} [\cf \cite{SchwedeTuckerTestIdealFiniteMaps}]
\label{prop.TransformationRuleOfTestIdeals}
Given a finite map $f \: Y \to X$ of normal $F$-finite integral schemes and a $\bQ$-Cartier $\bQ$-divisor $\Gamma$ on $X$, we have
\[
\Tr_f \big(f_* \tau(\omega_Y; f^* \Gamma) \big) = \tau(\omega_X; \Gamma) \, \, .
\]
\end{proposition}
\begin{proof}
Without loss of generality, we may assume that $X = \Spec R$ and $Y = \Spec S$ are affine and that $\Gamma = t \Div(g)$ for some $g \in R$.

Let $\Phi_R^e : F^e_* \omega_R \to \omega_R$  and $\Phi_S^e : F^e_* \omega_S \to \omega_S$ be the corresponding traces of Frobenius and set $J_S = \Image(\tr_{f} \: f_{*}\omega_S \to \omega_R) \subseteq \omega_{R}$. Now, \autoref{lem.ConstructiveDescriptionOfTestIdeals} above implies that there exists an element $c \in R$ such that both
\[
\tau(\omega_R; \Gamma) = \sum_{e \geq 0} \Phi_R^e(F^e_* c g^{\lceil t(p^e - 1) \rceil} J_S)
\]
\[
\tau(\omega_S; f^* \Gamma) = \sum_{e \geq 0} \Phi_S^e(F^e_* c g^{\lceil t(p^e - 1) \rceil} \omega_S) \, \, .
\]
The idea for the remainder of the proof is to apply $\Tr_f(\blank)$ to $\tau(\omega_S, f^* \Gamma)$, noting that $\Tr_f ( f_{*} \Phi_S^e( \blank)) = \Phi_R^e (   F^e_* \Tr_f( \blank))$ since trace is compatible with composition (shown precisely in \eqref{eq:ultimatetracecommutes}). Therefore,
\begin{align*}
 \Tr_f(f_* \tau(\omega_S; f^* \Gamma)) & =  \Tr_f\left( f_{*}\left(\sum_{e \geq 0} \Phi_S^e(F^e_* c g^{\lceil t(p^e - 1) \rceil} \omega_S) \right) \right) \\
& =  \sum_{e \geq 0} \Tr_f\left( f_{*}\Phi_S^e(F^e_* c g^{\lceil t(p^e - 1) \rceil} f_{*}\omega_S) \right) \\
& =  \sum_{e \geq 0} \Phi_R^e \left( F^e_* \Tr_f( c g^{\lceil t(p^e - 1) \rceil} f_{*}\omega_S) \right) \\
& =  \sum_{e \geq 0} \Phi_R^e (F^e_* c g^{\lceil t(p^e - 1) \rceil} J_{S})\\
& =  \tau(\omega_R; \Gamma)
\end{align*}
which completes the proof.
\end{proof}
To reduce the main theorem of this paper to \autoref{thm.CharOfParTestSubmodule}, we need a variant of the cyclic covering construction, \textit{cf.} \cite{TomariWatanabeNormalZrGradedRingsAndNormalCyclicCovers} or \cite[Section 2.4]{KollarMori}.
\begin{lemma}
\label{lem.PullBackQCartierViaSeparable}
Suppose that $X$ is a normal integral scheme and $\Gamma$ is a $\bQ$-Cartier $\bQ$-divisor on $X$.  Then there exists a finite \emph{separable} map $g \: W \to X$ from a normal integral scheme $W$ such that $g^* \Gamma$ is a Cartier divisor.
\end{lemma}
\begin{proof}
We may assume that $X = \Spec R$ is affine and that $n \Gamma = \Div_X(f)$ for some non-zero non-unit $f \in R$.  We view $f \in K = K(R)$ and suppose that $\alpha$ is a root of the polynomial $x^n + fx + f$ in some separable finite field extension $L$ of $K$.
Let $S$ be the normalization of $R$ inside $L$ so that we have a module-finite inclusion $R \subseteq S$.  Set $\pi \: Y = \Spec S \to \Spec R = X$.  Further observe that $S$ contains $\alpha$ since $\alpha$ is integral over $R$.  Since $\alpha^n  = -f(\alpha + 1)$
% and so since $f$ is not a unit, neither is $\alpha$.
we have $\alpha, \alpha+1 \in \sqrt{\langle \alpha + 1 \rangle}$, so that $\alpha + 1$ is a unit.
Thus, $n\Div_Y(\alpha) = \Div_Y(f) = \pi^* n \Gamma$, and so $\pi^* \Gamma = \Div_Y(\alpha)$ is Cartier as desired.
\end{proof}

\begin{theorem}
\label{thm.MainThm}
Given a normal integral $F$-finite scheme $X$ with a $\bQ$-divisor $\Delta$ such that $K_X + \Delta$ is $\bQ$-Cartier, there exists a finite separable map $f \: Y \to X$ from a normal $F$-finite integral scheme $Y$ such that
\begin{equation}
\label{eq.TestIdealEqualsFiniteImage}
\tau(X; \Delta) = \Image \left( f_* \omega_Y( -\lfloor f^*(K_X + \Delta) \rfloor) \xrightarrow{\Tr_f} K(X) \right).
\end{equation}
Alternatively, if $X$ is of finite type over a $F$-finite (respectively perfect) field, one may take $f : Y \to X$ to be a regular (respectively separable) alteration.
\end{theorem}

Before proving this theorem, we state several corollaries.

\begin{corollary}
\label{cor.TestIdealConstructedForManyAlterations}
Assume that $X$ is a normal variety over an $F$-finite field $k$. If $(X,\Delta)$ is a log-$\bQ$-Gorenstein pair, then using the images of the trace map as in \autoref{prop.TraceInGenerality}, we have
\begin{equation}
\label{eq.TestIdealIsBigIntersection}
\tau(X; \Delta) = \bigcap_{f : Y \to X} \Image\left(\xymatrix{ \myH^{\dim Y - \dim X} \myR f_* \omega_Y( -\lfloor f^*(K_X + \Delta) \rfloor) \xrightarrow{\Tr_f} K(X) }\right)
\end{equation}
where $f \: Y \to X$ ranges over all maps from a normal variety $Y$ that are either:
\begin{enumerate}
\item finite dominant maps
\item finite separable dominant maps
\item alterations (\ie generically finite proper dominant maps)
\item regular alterations
\item proper dominant maps
\item proper dominant maps from regular schemes
\end{enumerate}
or, if additionally $k$ is perfect,
\begin{enumerate}
\item[(g)] regular separable alterations.
\end{enumerate}
Furthermore, in all cases the intersection stabilizes (\ie is equal to one of its members).
\end{corollary}

\begin{corollary}
\label{cor.TestIdealConstructedForProperDominantMaps}
Assume $X$ is a normal integral $F$-finite scheme with a $\bQ$-divisor $\Delta$ such that $K_X + \Delta$ is $\bQ$-Cartier.
Then using the images of the trace map as in \autoref{prop.TraceInGenerality}, we have
\begin{equation}
\label{eq.TestIdealEqualsComplexImage}
\tau(X; \Delta) = \bigcap_{f : Y \to X} \Image\left(\xymatrix{ \myH^{- \dim X} \myR f_* \omega_Y^{\mydot} ( -\lfloor f^*(K_X + \Delta) \rfloor) \xrightarrow{\Tr_{f^{\mydot}}} K(X) }\right)
\end{equation}
where the intersection runs over all proper dominant maps $f \: Y \to X$ from a normal integral scheme $Y$ such that $f^* (K_X + \Delta)$ is Cartier.  Furthermore, once again, the intersection stabilizes.
\end{corollary}

\begin{proof}[Proof of Theorem \ref{thm.MainThm}]
%Of course, the containment $\subseteq$ follows immediately from \autoref{prop.EasyContainmentViaAlteration}.
We make use the identification of $\tau(X; \Delta) = \tau(\omega_X; K_X + \Delta)$ from \autoref{lem:testidealmodulecorresp} and will prove the statement for the latter.
For \eqref{eq.TestIdealEqualsFiniteImage}, first take a separable finite map $f' \: X' \to X$ such that $\Gamma := f'^* (K_X + \Delta)$ is Cartier by \autoref{lem.PullBackQCartierViaSeparable}.  By \autoref{prop.TransformationRuleOfTestIdeals},
\[
\Tr_{f'}\left(f'_* \tau(\omega_{X'}; \Gamma) \right) = \tau(\omega_X; K_X + \Delta) = \tau(X; \Delta).
\]
Now, using \autoref{thm.CharOfParTestSubmodule} we may fix a separable finite map $h \: Y \to X'$ such that we have
$\tau(\omega_{X'}) = \Image( \tr_{h} \: h_{*}\omega_{Y} \to \omega_{X'})$.
The projection formula then gives the equality
$\tau(\omega_{X'}; \Gamma) = \Image(\Tr_{h}(h_{*}\omega_Y(-h^* \Gamma) \to \omega_{X'})$, and \eqref{eq.TestIdealEqualsFiniteImage} now follows after applying $\Tr_{f'}$.

For the remaining statement when $X$, if we are given a composition of alterations $f \circ g : Z \xrightarrow{g} Y \xrightarrow{f} X$, it follows that
\[
\begin{array}{l}
\Image( f_* g_* \omega_Z( -\lfloor g^*f^*(K_X + \Delta) \rfloor ) \to[\Tr_{f \circ g}] K(X) ) \\ \qquad \qquad \qquad \qquad \subseteq \Image( f_* \omega_Y( -\lfloor f^*(K_X + \Delta) \rfloor) \to[\tr_{f}] K(X)) \, \, .
\end{array}
\]
% \[
% \Tr_{f \circ g} \left( f_* g_* \omega_Z( -\lfloor g^*f^*(K_X + \Delta) \rfloor) \right) \subseteq \Tr_f \left( f_* \omega_Y( -\lfloor f^*(K_X + \Delta) \rfloor) \right).
% \]
Thus, the second statement immediately follows from the first by taking a further regular (separable) alteration using \cite[Theorem 4.1]{de_jong_smoothness_1996}.
\end{proof}

\begin{proof}[Proof of Corollary \ref{cor.TestIdealConstructedForManyAlterations}]
For equation \eqref{eq.TestIdealIsBigIntersection}, we have the containment $\tau(X; \Delta) \subseteq$ from \autoref{prop.EasyContainmentViaAlteration} or \autoref{prop.EasyContainmentViaGeneralMaps}.  The result then follows from equation \eqref{eq.TestIdealEqualsFiniteImage}.
\end{proof}

\begin{proof}[Proof of Corollary \ref{cor.TestIdealConstructedForProperDominantMaps}]
Finally, for equation \eqref{eq.TestIdealEqualsComplexImage} we still have the containment $\tau(X; \Delta) \subseteq $ from \autoref{prop.EasyContainmentViaGeneralMaps}.  On the other hand, if $Y$ has the same dimension as $X$, then it is readily seen from the spectral sequence argument used in the proof of \autoref{thm.TraceNonzero} that
\[
\Tr_{f^{\mydot}}\left( \myH^{- \dim X} \myR f_* \omega_Y^{\mydot} ( -\lfloor f^*(K_X + \Delta) \rfloor) \right) = \Tr_{f} \left( \myH^{\dim Y - \dim X} \myR f_* \omega_Y( -\lfloor f^*(K_X + \Delta) \rfloor) \right) \, \, .
\]
so that equality holds in \eqref{eq.TestIdealEqualsComplexImage} for $f \: Y \to X$ as in \autoref{thm.MainThm}.
\end{proof}

\begin{remark}
If $R$ is an $F$-finite $\bQ$-Gorenstein splinter ring (i.e. any module-finite extension $R \subseteq S$ splits as a map of $R$-modules), then \autoref{cor.TestIdealConstructedForManyAlterations} above gives that $\tau(X) = \O_X$, implying that $R$ is strongly $F$-regular (since by $\tau(X)$ we always mean the big test ideal).  This recovers the main result of \cite{SinghQGorensteinSplinterRings}.
\end{remark}

\section{Nadel-type vanishing up to finite maps}%\ok{--}{--}{--}
\label{sec.NadelVanishing}
Among the most sorely missed tools in positive characteristic birational algebraic geometry (in comparison to characteristic zero) are powerful cohomology vanishing theorems. Strong additional assumptions (\textit{e.g.} lifting to the second Witt vectors) are required to recover the most basic version of Kodaira vanishing, and even under similar assumptions the most powerful variants (\textit{e.g.} Kawamata-Viehweg or Nadel-type vanishing) cannot be proven. By applying the results and ideas of B. Bhatt's dissertation (see also \cite{BhattDerivedDirectSummand}), we derive here variants of Nadel-type vanishing theorems. These are strictly weaker than what one would hope for as we only obtain the desired vanishing after a finite covering. Notably however, we need not require a W2 lifting hypothesis.

Before continuing, we recall the following well known Lemma.
\begin{lemma}[\cf \cite{BhattThesis}]
 Suppose that $X$ is a Noetherian scheme and that we have a map of objects $f : A^{\mydot} \to B^{\mydot}$ within $D^{\geq 0}_{\coherent}(X)$.  Then $f$ factors through $\myH^0 B^{\mydot}$ if and only if ${\boldsymbol{\tau}}_{> 0}(f)$ is the zero map.
\end{lemma}
\begin{proof}
 Certainly if $f$ factors through $\myH^0 B^{\mydot}$, then ${\boldsymbol{\tau}}_{> 0}(f) = 0$ since ${\boldsymbol{\tau}}_{> 0}(\myH^0 B^{\mydot}) = 0$.  For the converse direction, suppose ${\boldsymbol{\tau}}_{>0}(f) = 0$.
Consider the diagram
\[
 \xymatrix{
\myH^0 A^{\mydot} \ar[d] \ar[r] & A^{\mydot} \ar[r] \ar[d]  & {\boldsymbol{\tau}}_{>0} A^{\mydot} \ar[r]^-{+1} \ar[d]^0 &\\
\myH^0 B^{\mydot} 	 \ar[r] & B^{\mydot} \ar[r]         & {\boldsymbol{\tau}}_{>0} B^{\mydot} \ar[r]^-{+1}  &
}
\]
Since $A^{\mydot} \to B^{\mydot} \to {\boldsymbol{\tau}}_{>0} B^{\mydot}$ is zero, we also have the following diagram of objects in $D^{\geq 0}_{\coherent}(X)$
\[
 \xymatrix{
A^{\mydot} \ar@{.>}[d] \ar@{.>}[r]^{\id} & A^{\mydot} \ar[r] \ar[d]  & 0 \ar[r]^-{+1} \ar[d] &\\
\myH^0 B^{\mydot} 	 \ar[r] & B^{\mydot} \ar[r]         & {\boldsymbol{\tau}}_{>0} B^{\mydot} \ar[r]^-{+1}  &
} \, \, .
\]
\end{proof}

We begin with a lemma which can be viewed as a kind of ``Grauert-Riemenschneider vanishing \cite{GRVanishing} up to finite maps.'' Using the notation from below, note that in the special case where $W$ is smooth (or a tame quotient singularity) it has been shown only recently by A. Chatzistamatiou and K. R\"ulling \cite{ChatzistamatiouRullingHigherDirectImages} that $\myH^i \myR\pi_*\omega_X$ is zero for $i > 0$.
\begin{lemma}
\label{lem.GRVanishingUptoFiniteMaps}
 Suppose that $\pi \: X \to W$ is an alteration between integral schemes of characteristic $p > 0$.  Then there exists a finite map $g \: U \to X$ such that the trace map
\begin{equation}
\label{eq.GRDotCohomologyVanishing}
{\boldsymbol{\tau}}_{> -\dim X} \myR (\pi \circ g)_* \omega_U^{\mydot} \to {\boldsymbol{\tau}}_{> -\dim X} \myR \pi_* \omega_W^{\mydot}
\end{equation}
is zero and furthermore that the trace map
\begin{equation}
\label{eq.GRNotDotCohomologyVanishing}
{\boldsymbol{\tau}}_{> 0} \myR (\pi \circ g)_* \omega_U \to {\boldsymbol{\tau}}_{> 0} \myR \pi_* \omega_W
\end{equation}
is zero.  As a consequence $\myH^i \myR (\pi \circ g)_* \omega_U \to \myH^i \myR \pi_* \omega_X$ is zero for all $i > 0$.
\end{lemma}
\begin{proof}
 It is harmless to assume that $\pi$ is birational (simply take the normalization of $W$ inside the fraction field of $X$) and also that $W$ is affine.

First, choose a finite cover $a : X' \to X$ such that $a : \omega_{X'}^{\mydot} \to \omega_X^{\mydot}$ factors through $\omega_X[\dim X]$ by \cite[Proposition 5.4.2]{BhattThesis}.  Set $X' \to W' \to W$ to be the Stein factorization of $\pi \circ a$ (thus $\ell : W' \to W$ is finite).
By \cite[Theorem 5.0.1]{BhattThesis}, there exists a finite cover $b : \overline{X} \to X'$ from a normal $\overline{X}$ such that $\myR (\pi \circ a)_* \O_{X'} \to \myR (\pi \circ a \circ b)_* \O_{\overline{X}}$ factors through $(\pi \circ a \circ b)_* \O_{\overline{X}}$.  Set $\overline{X} \to \overline{W}$ to be the Stein-factorization of $\pi \circ a \circ b$ (thus $m : \overline{W} \to W'$ is finite and $(\pi \circ a \circ b)_* \O_{\overline{X}} = (l \circ m)_* \O_{\overline{W}}$).  Choose a further cover $n : \tld W \to \overline{W}$ such that ${\boldsymbol{\tau}}_{> -\dim X} (n_* \omega_{\tld{W}}^{\mydot}) \to {\boldsymbol{\tau}}_{> -\dim X} (\omega_{\overline{W}}^{\mydot})$ is zero by \cite[Proposition 5.4.2]{BhattThesis}.  By making $\tld{W}$ larger if necessary, we additionally assume that $c : \tld X \to \overline{X}$, the normalization of $\overline{X}$ in the fraction field of $\tld{W}$, satisfies the condition that $\myR (\pi \circ a \circ b)_* \O_{\overline{X}} \to \myR (\pi \circ a \circ b  \circ c)_* \O_{\tld X}$ factors through $(\pi \circ a \circ b \circ c)_* \O_{\tld X} \cong (l \circ m \circ n)_* \O_{\tld W}$ again using \cite[Theorem 5.0.1]{BhattThesis}.
\[
 \xymatrix{ \tld{X} \ar[r]^{c} \ar[d] & \overline{X} \ar[d] \ar[r]^b & X' \ar[d] \ar[r]^a & X \ar[d]^{\pi} \\
	    \tld{W} \ar[r]_{n} & \overline{W} \ar[r]_m & W' \ar[r]_{\ell} & W
}
\]
Putting this all together, we have the following factorization of $\myR \pi_* \O_{X} \to \myR \pi_* \O_{\tld X}$:
\begin{multline*}
\myR \pi_* \O_{X} \to \myR (\pi \circ a)_* \O_{X'} \to (\pi \circ a \circ b)_* \O_{\overline{X}} \\ \to \myR (\pi \circ a \circ b)_* \O_{\overline{X}}  \to (\pi \circ a \circ b \circ c)_* \O_{\tld X}  \to \myR (\pi \circ a \circ b \circ c)_* \O_{\tld X}.
\end{multline*}
We now note that the term $\myR (\pi \circ a \circ b)_* \O_{\overline{X}}$ in the factorization can be removed yielding:
\[
\myR \pi_* \O_{X} \to \myR (\pi \circ a)_* \O_{X'} \to (\pi \circ a \circ b)_* \O_{\overline{X}} \to (\pi \circ a \circ b \circ c)_* \O_{\tld X}  \to \myR (\pi \circ a \circ b \circ c)_* \O_{\tld X}.
\]
which by factoring along the lower part of the above diagram yields
\[
\myR \pi_* \O_{X} \to \myR (\pi \circ a)_* \O_{X'} \to (\ell \circ m)_* \O_{\overline{W}}  \to (\ell \circ m \circ n)_* \O_{\tld W}  \to \myR (\pi \circ a \circ b \circ c)_* \O_{\tld X}.
\]
Now we apply the functor $\myR \sHom_{X}(\blank, \omega_W^{\mydot})$ to this factorization and obtain:
\[
\myR \pi_* \omega_{X}^{\mydot} \ot \myR (\pi \circ a)_* \omega_{X'}^{\mydot} \ot (\ell \circ m)_* \omega_{\overline{W}}^{\mydot}  \ot[\alpha] (\ell \circ m \circ n)_* \omega_{\tld W}^{\mydot} \ot \myR (\pi \circ a \circ b \circ c)_* \omega_{\tld X}^{\mydot}
\]
Note we don't need $\myR$ on $\ell$, $m$ and $n$ since they are finite.  Setting $U = \tld X$, the first statement of the theorem, \eqref{eq.GRDotCohomologyVanishing} now follows since $\boldsymbol{\tau}_{>-\dim W} (\alpha) = 0$ based on our choice of $n$.  However,
\[
\myR \pi_* \omega_{X}^{\mydot} \ot \myR (\pi \circ a)_* \omega_{X'}^{\mydot}
\]
factors through $\myR \pi_* \omega_X[\dim X]$.  Precomposing with the natural map $\myR (\pi \circ a \circ b \circ c)_* \omega_{\tld X}^{\mydot} \leftarrow \myR (\pi \circ a \circ b \circ c)_* \omega_{\tld X}[\dim X]$ and taking cohomology yields \eqref{eq.GRNotDotCohomologyVanishing}.
\end{proof}

\begin{theorem}
\label{thm.NadelInCharP}
Suppose that $\pi \: X \to S$ is a proper morphism of $F$-finite integral schemes of characteristic $p > 0$ with $X$ normal.  Further suppose that $L$ is a Cartier divisor on $X$ and that $\Delta$ is a $\bQ$-divisor on $X$ such that $L - (K_X + \Delta)$ is a $\pi$-big and $\pi$-semi-ample $\bQ$-Cartier $\bQ$-divisor on $X$.  Then there exists a finite surjective map $f \: Y \to X$ from a normal integral $F$-finite scheme $Y$ such that:
\begin{enumerate}
\item  The natural trace map $$f_* \O_Y(\lceil K_Y + f^*(L - (K_X + \Delta)) \rceil) \to \O_X(\lceil K_X +  L - (K_X + \Delta) \rceil)$$ has image $\tau(X; \Delta) \tensor \O_X(L)$.
\item  The induced map on cohomology $$\myH^i \myR(\pi \circ f)_* \O_Y(\lceil K_Y + f^*(L - (K_X + \Delta)) \rceil) \to \myH^i \myR \pi_* \tau(X; \Delta) \tensor \O_X(L)$$ is zero for all $i > 0$.
\end{enumerate}
\end{theorem}
\begin{proof}[Proof (cf. \cite{BhattThesis})]
Certainly by \autoref{thm.MainThm} we can assume that (a) holds for some surjective finite map $f' \: Y' \to X$ (and every further finite map).  On $Y'$, we may also assume that $f'^* \Delta$ is integral and $f'^*(K_X + \Delta)$ is Cartier, and we may further assume that $\O_{Y'}\left(f'^*(L - (K_X +\Delta))\right)$ is the pull-back of a line bundle $\sL$ via some map $\pi \: Y' \to W$ over $S$ such that $\sL$ is ample over $S$.  Note, $f'^* (L - (K_X + \Delta))$ is still big so we may assume that $W$ has the same dimension as $Y'$ (and thus also the same dimension as $X$).
\[
\xymatrix{
Y' \ar[r]^{f'} \ar[d]_{\pi} & X \ar[d] \\
W \ar[r]_{\rho} & S
}
\]
Since we only need now to prove (b), replacing $X$ by $Y'$ we may assume that $K_X + \Delta$ is Cartier and that $\O_X\left( L - (K_X +\Delta) \right)$ is the pull-back of some relatively ample line bundle $\sL$ via an alteration $g : X \to W$ over $S$ with structural map $\rho \: W \to S$.
\[
\xymatrix{
Y \ar[r]^-h & X := Y' \ar[d]_{g} \ar[dr]^{\pi} \\
& W \ar[r]_{\rho} & S
}
\]

By \autoref{lem.GRVanishingUptoFiniteMaps}, there exists a finite cover $h : Y \to X$ such that $\myR (g \circ h)_* \omega_Y \to \myR g_* \omega_X$ factors through $g_* \omega_X$.  Choose $n_0 > 0$ such that $\myH^i \myR \rho_* \left( g_* \omega_X \tensor \sL^n \right) = 0$ for all $i > 0$ and all $n > n_0$ (by Serre vanishing).  Also choose an integer $e > 0$ such that $p^e > n_0$.  Consider now $\xymatrix{Y \ar[r]^h & X \ar[r]^{g} & W \ar[r]^{F^e} & W \ar[r]^{\rho} & S}$ which can also be expressed as
$Y \to[h] X \to[g] W  \to[\rho]  S \to[F^e] S$ and $Y \to[h]  X \to[g]  W \to[F^e] W \to[\rho] S$.  Finally consider the factorization:
\begin{equation*}
\begin{split}
\myH^i \myR (\rho \circ F^e \circ g \circ h)_*  &(\omega_Y \tensor h^*g^*(F^e)^* \sL)
 = \myH^i \myR (\rho \circ F^e)_* \left(  \sL^{p^e} \tensor \myR (g \circ h)_* \omega_Y \right) \\
&\to  F^e_* \myH^i \left( \myR (\rho)_* (\sL^{p^e} \tensor g_* \omega_X) \right) \\
&\to  F^e_* \myH^i \left( \myR (\rho)_* (\sL^{p^e} \tensor \myR g_* \omega_X) \right)
=  \myH^i \left( (\myR \rho_*) F^e_*(\sL^{p^e} \tensor \myR g_* \omega_X) \right) \\
&\to  \myH^i \left( (\myR \rho_*) (\sL \tensor \myR g_* F^e_* \omega_X) \right) \\
&\to  \myH^i \left( (\myR \rho_*) (\sL \tensor \myR g_* \omega_X) \right) =  \myH^i \myR \pi_* \left(  g^* \sL \tensor \omega_X \right).
\end{split}
\end{equation*}
This map is zero since the second line is zero by construction.
\end{proof}

\begin{corollary}
\label{cor.NadelVanishingForVarieties}
Let $X$ be a projective variety over an $F$-finite field $k$.  Suppose that $L$ is a Cartier divisor on $X$, and that $\Delta$ is a $\bQ$-divisor such that $L - (K_X + \Delta)$ is a big and semi-ample $\bQ$-Cartier $\bQ$-divisor.  Then there exists a finite surjective map $f \: Y \to X$ from a normal variety $Y$ such that:
\begin{enumerate}
\item  The natural trace map $$f_* \O_Y(\lceil K_Y + f^*(L - (K_X + \Delta) \rceil) \to \O_X(\lceil K_X +  L - (K_X + \Delta) \rceil)$$ has image $\tau(X; \Delta) \tensor \O_X(L)$.
\item  The induced map on cohomology
\[
H^i\big(Y, \O_Y(\lceil K_Y + f^*(L - (K_X + \Delta)) \rceil) \big) \to H^i\big(X, \tau(X; \Delta) \tensor \O_X(L)\big)
\]
is zero for all $i > 0$.
\end{enumerate}
\end{corollary}

\section{Transformation rules for test ideals under alterations}%\ok{--}{--}{--}
\label{sec.TransformationRulesForTestIdeals}

The fact that the test ideal can be computed via alterations suggests a transformation rule for test ideals under alterations.  We derive this transformation rule in this section.
We first state a definition.

\begin{definition}
\label{def.TauCohomology}
Suppose that $X$ is a normal variety over a field $k$ and $\Gamma$ is a $\bQ$-Cartier $\bQ$-divisor. For example, one might take $\Gamma = L - (K_X + \Delta)$ where $L$ is a Cartier divisor and $K_X + \Delta$ is $\bQ$-Cartier. We define
\[
\tauCohomology^0(X, \Gamma) \colonequals \bigcap_{f \: Y \to X} \Image \Big(H^0(Y, \O_{Y}\big(\lceil K_Y + f^*\Gamma\rceil)\big) \to[\tr_{f}] H^0\big(X, \O_X(\lceil K_X + \Gamma\rceil)\big)\Big)
\]
where $f$ runs over all finite dominant maps $f \:Y \to X$ such that $Y$ is normal and equidimensional.

Alternately, if $X$ is any (not necessarily normal or even reduced) $F$-finite $d$-dimensional equidimensional scheme of finite type over a field $k$ and $\sL$ is any line bundle on $X$, then we define
\[
\tauCohomology^0(X, \sL) \colonequals \bigcap_{f \: Y \to X} \Image\Big( H^0\big(Y, \omega_Y \tensor f^* \sL\big) \to[\tr_{f}] H^0\big(X, \omega_X \tensor \sL\big) \Big)
\]
where $f$ runs over all finite dominant maps $f : Y \to X$ such that $Y$ is normal and equidimensional of dimension $d$. In both cases the maps are induced by the trace map as described in \autoref{prop.TraceInGenerality}.
\end{definition}

\begin{remark}
We expect that the reader has noticed the upper-script $0$ in the definition of $\tauCohomology^0(X, \Gamma)$.  We include this for two reasons:  \begin{itemize}
\item[(i)]  It serves to remind the reader that $\tauCohomology^0(X, \Gamma)$ is a submodule of $H^0\left(X, \O_X(\lceil K_X + \Gamma\rceil)\right)$.
\item[(ii)]  In the future, it might be reasonable to extend this definition to higher cohomology groups.  For example, \autoref{thm.NadelInCharP} is a vanishing theorem for appropriately defined higher $\tauCohomology^i(X, \Gamma)$.
\end{itemize}
\end{remark}

First we make a simple observation:
\begin{lemma}
\label{lem.TauTransformationRuleForFiniteMaps}
For any finite dominant map $f \: Y \to X$ between proper normal varieties over a field $k$ with $\Gamma$ a $\bQ$-Cartier $\bQ$-divisor on $X$, then $\tauCohomology^0(Y, f^*\Gamma)$ is sent onto $\tauCohomology^0(X, \Gamma)$ via the trace map
\[
\beta : H^0\big(Y, \O_{Y}(\lceil K_Y + f^*\Gamma \rceil)\big) \to H^0\big(X, \O_X(\lceil K_X + \Gamma )\rceil) \big).
\]
\end{lemma}
\begin{proof}
If $Y$ is proper over a field $k$, then $H^0\big(Y, \O_{Y}(\lceil K_Y + f^*\Gamma \rceil)\big)$ is a finite dimensional vector space, and so $\tauCohomology^0(Y, f^*\Gamma)$ is the image of a single map
\[
H^0\big(Y', \O_{Y'}(\lceil K_{Y'} + g^* f^* \Gamma \rceil)\big) \to H^0\big(Y, \O_Y(\lceil K_Y + f^* \Gamma \rceil) \big) \, \, ,
\]
for some finite cover $g \: Y' \to Y$.
Composing with the map to $H^0\big(X, \O_X(\lceil K_X + \Gamma \rceil) \big)$ yields the inclusion $\tauCohomology^0(X,\Gamma) \subseteq \beta\left( \tauCohomology^0(Y,f^*\Gamma)\right).$

For the other inclusion, we simply notice that given any finite dominant map $h : W \to X$, we can find a finite dominant map $U \to X$ which factors through both $h$ and $f$.
\end{proof}

\begin{lemma}
\label{lem.TauCanBeDefinedUsingDualizing}
Suppose that $X$ is as in \autoref{def.TauCohomology}.  Observe that $H^0(X, \omega_X \tensor \sL) \cong \bH^0(X, \omega_X^{\mydot}[-\dim X] \tensor \sL)$.  Furthermore, $\tauCohomology^0(X, \omega_X \tensor \sL)$ is:
\[
\bigcap_{f : Y \to X} \Image\left( \bH^0(Y, \omega_Y^{\mydot}[-\dim X] \tensor f^* \sL) \to[\tr_{f^{\mydot}}] \bH^0(X, \omega_X^{\mydot}[-\dim X] \tensor \sL)
\right).
\]
where the intersection runs over all finite dominant maps $f : Y \to X$ such that $Y$ is normal and equidimensional.
\end{lemma}
\begin{proof}
First note that the isomorphism $H^0(X, \omega_X \tensor \sL) \cong \bH^0(X, \omega_X^{\mydot}[-\dim X] \tensor \sL)$ follows from analyzing the spectral sequence computing $\bH^0(X, \omega_X^{\mydot}[-\dim X] \tensor \sL)$.
The second statement then follows immediately.
%The intersection above contains $\tauCohomology^0(X, \omega_X \tensor \sL)$ since we have a natural map:
%\[ H^0(Y, \omega_Y \tensor f^* \sL) \to \bH^0(Y, \omega_Y^{\mydot}[-\dim X] \tensor f^* \sL). \]
%On the other hand by \cite[Proposition 5.4.2]{BhattThesis}, given any $Y$, we may find a further cover $g : Z \to Y$ such that $\omega_Z^{\mydot} \to \omega_Y^{\mydot}$ factors through $\omega_Y[\dim X]$ and so the other containment also follows.
\end{proof}

Mimicking \autoref{lem.TauTransformationRuleForFiniteMaps}, we also obtain:
\begin{lemma}
\label{lem.TauTransformationRuleForFiniteMapsDualizing}
For any finite dominant map $f \: Y \to X$ between proper equidimensional schemes of finite type over a field $k$ with $\sL$ a line bundle on $X$, then $\tauCohomology^0(Y, \omega_Y \tensor f^*\sL)$ is sent onto $\tauCohomology^0(X, \omega_X )$ via the trace map
\[
\bH^0\big(Y, \omega_Y^{\mydot}[-\dim X] \tensor f^* \sL \big) \to \bH^0\big(X, \omega_X^{\mydot}[-\dim X] \tensor \sL \big) \, \, .
\]
\end{lemma}
\begin{proof}
The proof is identical to the proof of \autoref{lem.TauTransformationRuleForFiniteMaps}.
\end{proof}

By \autoref{thm.MainThm}, if $X = \Spec R$ is affine and $L = 0$, then $\tauCohomology^0(X, L - (K_X + \Delta))$ is just the global sections of $\tau(X; \Delta)$.
Inspired by this, we demonstrate that we may also use alterations in order to compute $\tauCohomology^0(X, L - (K_X + \Delta))$.

First, suppose that $f \: Y \to X$ is an alteration and recall from \autoref{prop.TraceInGenerality} that we have a map
\[
  \Tr_f \: f_*\omega_Y(-\lfloor f^*(K_X+\Delta) \rfloor) \to \omega_X(-\lfloor K_X+\Delta \rfloor) \, \, .
\]
Twisting by a Cartier divisor $L$ and taking cohomology leads us to a map
\begin{align}\label{eq.TraceOnCohomology}
\Psi : H^0\left(Y, \omega_Y(\lceil f^*(L - K_X - \Delta) \rceil)\right) & = H^0\left(X, f_*\omega_Y(\lceil f^*(L - K_X-\Delta) \rceil) \right)\\
& \to H^0\left(X, \omega_X(\lceil L - K_X-\Delta \rceil) \right) \nonumber \, \, .
\end{align}

\begin{theorem}
\label{thm.tauCohomologyViaAlterations}
Suppose that $X$ is a $F$-finite normal variety over $k$ and that $\Delta$ is a $\bQ$-divisor such that $K_X + \Delta$ is $\bQ$-Cartier.  Finally set $L$ to be any Cartier divisor. Then
\[
\begin{split}
\tauCohomology^0&(X, L - K_X - \Delta ) \\
&= \bigcap_{f : Y \to X} \Image \Big(H^0\left(Y, \omega_Y(\lceil f^*(L - K_X - \Delta )\rceil)\right)
\to[\tr_{f}] H^0\left(X, \omega_X(\lceil L - K_X - \Delta \rceil)\right) \Big)
\end{split}
\]
where $f$ runs over all alterations $f \:Y \to X$ and the maps in the intersection are as in~\eqref{eq.TraceOnCohomology}.
\end{theorem}
\begin{proof}
Certainly we have the containment $\tauCohomology^0(X, L - (K_X + \Delta)) \supseteq \bigcap_{f \: Y \to X}\left( \ldots \right)$ since this latter intersection intersects more modules than the one defining $\tauCohomology^0(X, L - (K_X + \Delta))$.  We need to prove the reverse containment.

Fix an alteration $f \: Y \to X$.  Set $Y \to[\alpha] W = {\sheafspec} (f_* \O_{Y}) \to[\beta] X$ to be the Stein factorization of $f$.  Certainly
\[
\tau(\omega_{W}, \beta^*(L - K_X - \Delta)) \subseteq \alpha_* \O_{Y}(\lceil K_{Y} + f^* (L - K_X - \Delta)\rceil).
\]
Choose a finite cover $h \: U \to W$ such that $\Tr_h \: \O_U(\lceil K_U - h^* \beta^*(L - K_X - \Delta) \rceil ) \to \O_W(\lceil K_W - \beta^*(L - K_X - \Delta)\rceil)$ has image $\tau(\omega_{W}, \beta^*(L - K_X - \Delta))$.
\[
\xymatrix{
& Y \ar[d]_{\alpha} \ar[dr]^f \\
U \ar[r]_h & W \ar[r]_{\beta} & X
}
\]
We now have the following factorization:
\[
\begin{split}
H^0(U, \O_U(\lceil K_U + h^* &\beta^*(L - K_X - \Delta) \rceil)) \to H^0(X, \beta_* \tau(\omega_{W}, \beta^*(L - K_X - \Delta))) \\
& \to H^0(X, \beta_* \alpha_* \O_{Y}(\lceil K_{Y} + f^* (L - K_X - \Delta)\rceil))\\
& = H^0(X, f_* \O_{Y}(\lceil K_{Y} + f^* (L - K_X - \Delta)\rceil)) \\
& = H^0(Y, \omega_Y(\lceil f^*(L - K_X - \Delta))\rceil) \\
& \to H^0(X, \omega_X(\lceil L - K_X - \Delta \rceil))
\end{split}
\]
And so we obtain the desired inclusion.
\end{proof}

\begin{theorem}
\label{thm.TransformationOfTestIdealsUnderDominantMaps}
Suppose that $f : Z \to X$ is an alteration between normal $F$-finite varieties over a field $k$, $\Delta$ is a $\bQ$-divisor on $X$ such that $K_X + \Delta$ is $\bQ$-Cartier and $L$ is any Cartier divisor on $X$.  Additionally suppose that either $X$ is affine or proper over $k$.

Then
\[
\Psi\Big( \tauCohomology^0 \big(Z, f^*(L - (K_X + \Delta))\big) \Big) = \tauCohomology^0\big(X, L - (K_X + \Delta)\big)
\]
where $\Psi$ is the map described in \eqref{eq.TraceOnCohomology}.

In particular, if $f$ is proper and birational and $X$ is affine, then this yields a transformation rule for the test ideal under proper birational morphisms.
\end{theorem}
\begin{proof}
Certainly $\tauCohomology^0(X, L - (K_X + \Delta)) \supseteq \Psi\big(\tauCohomology^0 (Z, f^*(L - (K_X + \Delta)))\big)$ in either case by \autoref{thm.tauCohomologyViaAlterations}.

If $X$ is proper over $k$, then so is $Z$ and so $H^0(Z, \O_{Z}(\lceil K_Z + f^*(L - (K_X +  \Delta))\rceil))$ is a finite dimensional $k$-vector space.  It follows that $\tauCohomology^0 (Z, f^*(L - (K_X + \Delta)))$ is the image of a single map
\[
H^0\left(Z', \O_{Z'}(\lceil K_{Z'} + g^* f^* (L - K_X - \Delta)\rceil)\right) \to[\Tr_{g}] H^0\left(Z, \O_{Z}(\lceil K_{Z} + f^* (L - K_X - \Delta)\rceil)\right)
\]
for some finite cover $g : Z' \to Z$.  However, $f \circ g : Z' \to X$ is also an alteration, and it follows that $\tauCohomology^0(X, L - (K_X + \Delta))$ is contained in the image of
\[
H^0\left(Z', \O_{Z'}(\lceil K_{Z'} + g^* f^* (L - K_X - \Delta)\rceil)\right) \to[\Tr_{f \circ g}] H^0\left(X, \O_{X}(\lceil K_{X} + L - K_X - \Delta \rceil)\right)
\]
whose image is clearly $\Tr_{f} \big(\tauCohomology^0 (Z, f^*(L - (K_X + \Delta))) \big)$.

On the other hand, suppose $X = \Spec R$ is affine and observe that, without loss of generality, we may assume that $K_X + \Delta$ is effective.  Set $S = H^0(Y, \O_Y)$ and define $Y \to[\alpha] X' = \Spec S \to[\beta] X$ to be the Stein factorization of $f$.
Now, the global sections $H^0(Z, \O_Z(\lceil K_Z + f^*(L - (K_X + \Delta))\rceil))$ can be identified with elements of $\omega_S(L)$ because we assumed that $K_X + \Delta \geq 0$.  In particular, we have $\tau(\omega_S, \beta^*(L - K_X - \Delta)) \subseteq H^0(Z, \O_Z(\lceil K_Z - f^*(L - (K_X + \Delta))\rceil))$.  But $\Tr_{\beta} \big( \tau(\omega_S, \beta^*(L - K_X - \Delta))\big) = \tau(R, L - K_X - \Delta)$ and the other containment follows.
\end{proof}

\begin{remark}
In order to generalize the above result to arbitrary schemes, it would be helpful to know that the intersection defining $\tauCohomology^0(X, \Gamma)$ stabilized in general.
\end{remark}

\section{Surjectivities on cohomology}%\ok{--}{--}{--}
\label{sec:surj-cohom}

In this section we show how the vanishing statements obtained in \cite{BhattThesis} and in \autoref{sec.NadelVanishing}, combined with the ideas of \autoref{sec.TransformationRulesForTestIdeals}, can be used to construct global sections of adjoint line bundles.  We are treating this current section as a proof-of-concept.  In particular, many of the statements can be easily generalized.  We leave the statements simple however in order to demonstrate the main ideas.
Consider the following prototypical application of Kodaira vanishing.

\begin{example}
Suppose that $X$ is a smooth projective variety in characteristic zero and that $D$ is an effective Cartier divisor on $X$.  Set $\sL$ to be an ample line bundle on $X$.  We have the following short exact sequence
\[
0 \to \omega_X \tensor \sL \to \omega_X(D) \tensor \sL  \to \omega_D \tensor \sL|_D \to 0 \, \, .
\]
Taking cohomology gives us
\[
0 \to H^0(X, \omega_X \tensor \sL ) \to H^0(X, \omega_X(D) \tensor \sL) \to H^0(D, \omega_D\tensor \sL|_D) \to H^1(X, \omega_X \tensor \sL ).
\]
Kodaira vanishing implies that $H^1(X, \omega_X \tensor \sL)$ is zero and so
\[
H^0(X, \omega_X(D)\tensor \sL) \to H^0(D, \omega_D\tensor \sL|_D )
\]
is surjective.
\end{example}

Consider now the same example (in characteristic zero) but do not assume that $X$ is smooth.

\begin{example}
\label{ex.KVVanishingImpliesSurjectivity}
Suppose that $X$ is a normal projective variety in characteristic zero and that $D$ is a reduced Cohen-Macaulay Cartier divisor on $X$.
These conditions are enough to imply that the natural map $\omega_X(D) \to \omega_D$ is surjective.

Set $\sL$ to be an ample (or even big and nef) line bundle on $X$.  Choose $\pi \: \tld X \to X$ to be a log resolution of $(X, D)$ and set $\tld D$ to be the strict transform of $D$ on $\tld X$.  We have the following diagram of exact triangles in $D^{b}_{\coherent}(X)$:
\[
\xymatrix{
0 \ar[r] & \pi_* \big( \omega_{\tld X}\tensor \pi^* \sL\big) \ar[r] \ar@{=}[d] & \pi_* \big(\omega_{\tld X}(\tld D) \tensor \pi^* \sL\big) \ar[r] \ar@{=}[d] & \pi_* \big( \omega_{\tld D}\tensor \pi^* \sL|_D\big) \ar[r] \ar@{=}[d] & 0\\
& \myR \pi_* \big( \omega_{\tld X}\tensor \pi^* \sL \big) \ar[r] \ar[d] & \myR \pi_* \big( \omega_{\tld X}(\tld D) \tensor \pi^* \sL \big) \ar[r] \ar[d] & \myR \pi_* \big( \omega_{\tld D} \tensor \pi^* \sL|_D \big) \ar[r]^-{+1} \ar[d] & \\
0 \ar[r] & \omega_X\tensor \sL  \ar[r] & \omega_X(D) \tensor \sL  \ar[r] & \omega_D\tensor \sL|_D  \ar[r] & 0
}
\]
The vertical equalities and the top right surjection are due Grauert-Riemenschneider vanishing \cite{GRVanishing}.
Because $X$ is not smooth, $H^1(X, \omega_X(L))$ is not necessarily zero, see for example \cite{ArapuraJaffeKodairaVanishingForSingular}.  However, $H^1(X, (\pi_* \omega_{\tld X}) \tensor \sL) = H^1(\tld X, \omega_{\tld X} \tensor \pi^* \sL )$ is zero by Kawamata-Viehweg vanishing since $\pi^* L$ is big and nef, \cite{KawamataVanishing, ViehwegVanishingTheorems}.  Thus we have the surjection:
\[
H^0(X, (\pi_* \omega_{\tld X}(\tld D)) \tensor \sL ) \to H^0(X, (\pi_* \omega_{\tld D}) \tensor \sL|_D) \, \, ,
\]
between submodules of $H^0(X, \omega_X(D)\tensor \sL)$ and $H^0(D, \omega_D\tensor \sL|_D )$.

Interestingly, $\pi_* \omega_{\tld D}$ is independent of the choice of embedding of $D$ into $X$ since $\pi_* \omega_{\tld D}$ is the multiplier module for $D$ by definition.  Even more, $H^0(D, \pi_* \omega_{\tld D} \tensor \sL|_D)$ only depends on the pair $(D, \sL|_D)$.

Furthermore, using the method of proof of \autoref{thm.TransformationRuleForMultProperDominant} below, it is easy to see that
\[
H^0(D, \pi_* \omega_{\tld D} \tensor \sL|_D) = \bigcap_{f \: E \to D} \Image\left( H^0(E, \omega_E \tensor f^*\sL|_D) \to H^0(D, \omega_D \tensor \sL|_D) \right)
\]
where the intersection runs over all regular alterations $f \: E \to D$.  In light of \autoref{thm.tauCohomologyViaAlterations} the subspace $H^0(D, \pi_* \omega_{\tld D} \tensor \sL|_D)$ may be viewed as an analog of $\tauCohomology^0(D, \omega_{\tld D} \tensor \sL|_D)$.  This inspires the remainder of the section.
\end{example}

In characteristic $p > 0$, Kodaira vanishing does not hold even on smooth varieties \cite{raynaud_contre-exemple_1978}.  However, we have the following corollary of \cite{BhattThesis}, \cf \autoref{cor.NadelVanishingForVarieties}.

\begin{theorem}
\label{thm.BaseSurjectivityOfVarieties}
Suppose that $D$ is a Cartier divisor on a normal proper $d$-dimensional variety $X$ and $\sL$ is a big and semi-ample line bundle on $X$. Using the natural map
\[
\gamma \: H^0(X, \omega_X(D) \tensor \sL) \to H^0(D, \omega_D \tensor \sL|_D)
\]
one has an inclusion
\[
\tauCohomology^0(D, \omega_D \tensor \sL|_D) \subseteq \gamma( \tauCohomology^0(X, \omega_X \tensor \sL(D)) ).
\]
In particular, if $\tauCohomology^0(D, \omega_D \tensor \sL|_D ) \neq 0$, then
$H^0(X, \omega_X(D) \tensor \sL) \neq 0.$
And if $\tauCohomology^0(D, \omega_D \tensor \sL|_D ) =  H^0(D, \omega_D \tensor \sL|_D)$ then $\gamma$ is surjective.
\end{theorem}
\begin{proof}
Using \autoref{lem.TauCanBeDefinedUsingDualizing}, set $f : Y \to X$ to be a finite cover of $X$ such that $\tauCohomology^0(X, \omega_X \tensor \sL(D))$ is equal to
\[
\Image\left( \bH^0(Y, \omega_Y^{\mydot}[-d] \tensor f^* \sL(D)) \to[\tr_{f^{\mydot}}] \bH^0(X, \omega_X^{\mydot}[-d] \tensor \sL(D)) \right)
\]
noting that $\bH^0(X, \omega_X^{\mydot}[-d] \tensor \sL(D)) \cong H^0(X, \omega_X \tensor \sL(D))$.
By \cite[Proposition 5.5.3]{BhattThesis}, there exists a finite cover $g \: Z \to Y$ such that $H^{d-1}(Y, f^* \sL^{-1}) \to H^{d-1}(Z, g^* f^* \sL^{-1})$ is the zero map.  Therefore the dual map $\bH^{1 - d}(Z, g^* f^* \sL \tensor \omega_Y^{\mydot}) \to \bH^{1-d}(Y, f^* \sL \tensor \omega_X^{\mydot})$ is zero.

Set $D_Y = f^* D$ and $D_Z = g^* f^* D$.  Note that $D_Y$ or $D_Z$ may not be normal or even reduced, even if $D$ is.  They are however equidimensional.  Then there is a map between long exact sequences:
\[
\tiny
\xymatrix@C=10pt{
0 \ar[r] & \bH^{-d}(Z, g^*f^* \sL \tensor \omega_Z^{\mydot}) \ar[r] \ar[d] & \bH^{-d}(Z, g^*f^* \sL \tensor \omega_Z^{\mydot}(D_Z) )\ar[r] \ar[d] & \bH^{1-d}(D_Z, g^* f^* \sL|_D \tensor \omega_{D_Z}^{\mydot} ) \ar[r] \ar[d]_{\beta} & \bH^{1 - d}(Z, g^* f^* \sL \tensor \omega_Z^{\mydot}) \ar[d]_0 \\
0 \ar[r] & \bH^{-d}(Y, f^* \sL \tensor \omega_Y^{\mydot}) \ar[r] \ar[d] & \bH^{-d}(Y, f^* \sL \tensor \omega_Y^{\mydot}(D_Y) )\ar[r] \ar[d]_{\nu} & \bH^{1-d}(D_Y, f^* \sL|_D \tensor \omega_{D_Y}^{\mydot} ) \ar[r]^-{\delta} \ar[d]_{\alpha} & \bH^{1 - d}(Y, f^* \sL \tensor \omega_Y^{\mydot}) \ar[d] \\
0 \ar[r] & \bH^{-d}(X, \sL \tensor \omega_X^{\mydot}) \ar[r] \ar@{=}[d] & \bH^{-d}(X, \sL \tensor \omega_X^{\mydot}(D) )\ar[r]_{\gamma} \ar@{=}[d] & \bH^{1-d}(D, \sL|_D \tensor \omega_{D}^{\mydot} ) \ar[r] \ar@{=}[d] & \bH^{1 - d}(X, \sL \tensor \omega_X^{\mydot})\\
0 \ar[r] & H^{0}(X, \sL \tensor \omega_X) \ar[r] & H^{0}(X, \sL \tensor \omega_X(D)) \ar[r]_{\gamma} & H^{0}(D, \sL|_D \tensor \omega_D)
}
\]
where the vertical equalities are obtained from the spectral sequences computing the middle lines.  Note we identify $f$ with its restriction $f|_{D_Y}$ and $g$ with $g|_{D_Z}$

Choose the $h : E \to D_Z$ to be normalization of the $(D_Z)_{\red}$ and notice we have a map
\[
\begin{split}
H^0(E, \omega_E \tensor h^* g^* f^* \sL|_D )& =  \bH^{1-d}(E, \omega_E^{\mydot} \tensor h^* g^* f^* \sL|_D )\\
& \to  \bH^{1-d}(D_Z, \omega_{D_Z}^{\mydot} \tensor g^* f^* \sL|_D ) \\
& \to[\alpha \circ \beta]  \bH^{1-d}(D, \omega_{D}^{\mydot} \tensor \sL|_D)=  H^0(D, \omega_D \tensor \sL|_D ).
\end{split}
\]
The image of this map contains $\tauCohomology^0(D, \omega_D \tensor \sL|_D)$, and thus the image of $\alpha \circ \beta$ also contains $\tauCohomology^0(D, \omega_D \tensor \sL|_D)$.
Therefore, if we view $d \in \tauCohomology^0(D, \omega_D \tensor \sM)$ as an element of $\bH^{1-d}(D, \sL|_D \tensor \omega_{D}^{\mydot} )$, it must have some pre-image $z \in \bH^{1-d}(D_Z, f^* \sL|_D \tensor \omega_{D_Z}^{\mydot} )$.  The diagram implies that $\delta(\beta(z)) = 0$.  Thus $\beta(z)$ is an image of some element in $y \in H^{0}(X, f^* \sL \tensor \omega_Y^{\mydot}(D_Y))$.  It follows that $\nu(y) \in \tauCohomology^0(X, \omega_X \tensor \sL(D))$ and so $d = \gamma(\nu(y))$ which completes the proof.
\end{proof}

\begin{remark}
If we knew that the intersection defining $\tauCohomology^0(X, L - K_X - \Delta)$ stabilized, then the previous result could be generalized to arbitrary equidimensional schemes (not just those which are of finite type over a field).  Even without this hypothesis, the argument above still implies that $\tauCohomology^0(D, \omega_D \tensor \sL|_D) \subseteq \gamma( H^0(X, \omega_X(D) \tensor \sL) ).$  The same statement holds for \autoref{thm.ExampleTheoremSurjectivity}.
\end{remark}

\begin{remark}
We expect that one can obtain more precise surjectivities involving characteristic $p > 0$ analogs of adjoint ideals.  In particular, $\tauCohomology^0(X, \omega_X \tensor \sL)$ is not the right analog of the term $H^0(X, \pi_* \omega_{\tld X}(\tld D))$ appearing in \autoref{ex.KVVanishingImpliesSurjectivity} above.
\end{remark}

We also show that this method can be generalized with $\bQ$-divisors.

\begin{theorem}
\label{thm.ExampleTheoremSurjectivity}
Suppose $X$ is a normal $F$-finite variety which is proper over a field $k$ and that $D$ is a Cartier divisor on $X$.  Additionally, suppose that $\Delta$ is a $\bQ$-divisor on $X$ with no common components with $D$ and such that $K_X + \Delta$ is $\bQ$-Cartier.  Finally suppose that $L$ is a Cartier divisor on $X$ such that $L - (K_X + D + \Delta)$ is big and semi-ample.  Then the natural map
\begin{multline*}
H^0(X, \O_X(\lceil K_X + D + L - (K_X + D + \Delta) \rceil)  =  H^0(X, \O_X(\lceil L - \Delta \rceil) \\
 \to[\gamma]  H^0(D, \O_D(\lceil K_D + L|_D - (K_D + \Delta|_D) \rceil))  =  H^0(D, \O_D(L|_D - \lfloor \Delta \rfloor|_D ).
\end{multline*}
yields an inclusion
\[
\tauCohomology^0(D, L|_D - (K_D + \Delta|_D)) \subseteq \gamma\left(\tauCohomology^0(X, D + L - (K_X + D + \Delta))\right)
\]
noting that $\tauCohomology^0(D, L|_D - (K_D + \Delta|_D)) \subseteq H^0(D, \O_D(\lceil L|_D - \lfloor \Delta \rfloor|_D \rceil))$.
\end{theorem}
\begin{proof}
Let us first point out that $\gamma$ is induced from the restriction map
\begin{multline*}
\O_X(\lceil K_X + D + L - (K_X + D + \Delta) \rceil) \\
\to  \O_D(\lceil K_X + D + L - (K_X + D + \Delta) \rceil|_D)
=  \O_D(K_D + L|_D - (K_D + \lfloor \Delta \rfloor|_D)).
\end{multline*}
Now choose a finite cover $h : W \to X$, with $W$ normal, such that $h^* (K_X + \Delta)$ is an integral Cartier divisor and set $D_W = h^* D$. Note $D_W$ is not necessarily normal or even reduced (it is however unmixed and thus equidimensional).  We have the diagram
\[
\xymatrix{
h_* \omega_W\big(h^* (D + L - (K_X + D + \Delta))\big) \ar[d] \ar[r] & h_* \omega_{D_W}\big(h^*(L - (K_X + D + \Delta))\big) \ar[d]_{\xi} \\
\omega_X(\lceil D + L - (K_X + D + \Delta) \rceil) \ar[r] & \omega_D(\lceil L|_D - (K_D + \Delta) \rceil|_D)
}
\]
of which we take global sections and then apply the method of Lemmas \ref{lem.TauTransformationRuleForFiniteMaps} and \ref{lem.TauTransformationRuleForFiniteMapsDualizing} to conclude that the image \mbox{$\xi \Big( \tauCohomology^0\big(D_W, \omega_{D_W} \tensor \O_W(h^*(L - (K_X + D + \Delta))) \big) \Big)$} is equal to
\begin{align}
\tauCohomology^0\big(D, L|_D - (K_D + \Delta|_D)\big) & \subseteq  H^0\big(D, \O_D(K_D + L|_D - (K_D + \Delta|_D))\big) \\
& \subseteq H^0\big(D, \omega_D(\lceil L|_D - (K_D + \Delta) \rceil|_D)\big).
\end{align}
Likewise $\tauCohomology^0(X, D + L -(K_X + D + \Delta))$ is the image of $\tauCohomology^0(W, h^*(D + L -(K_X + D + \Delta)))$.  Thus it is sufficient to show that via the map
\[
H^0\big(W, \omega_W(h^* (D + L - (K_X + D + \Delta))) \big) \to H^0
\big(D_W, \omega_{D_W}(h^*(L - (K_X + D + \Delta)))\big),
\]
each element of $\tauCohomology^0\big(D_W, \omega_{D_W} \tensor \O_W(h^*(L - (K_X + D + \Delta))) \big)$ is the image of an element of $\tauCohomology^0\big(W, \omega_W  \tensor \O_W(h^*(D + L -(K_X + D + \Delta)))\big)$.
We have just reduced to the setting of \autoref{thm.BaseSurjectivityOfVarieties} and the result follows.
\end{proof}

\section{Transformation rules for multiplier ideals}%\ok{--}{--}{--}
\label{sec.TransformationRulesForMultiplier}

It still remains to be proven that the multiplier ideal in characteristic zero can be characterized as in our Main Theorem, for which we need to explore the behavior of multiplier ideals under proper dominant maps.  We further analyze the behavior of multiplier ideals under alterations in arbitrary characteristic, which leads to an understanding of when (and why) the classical characteristic zero transformation rule \eqref{eq:multfiniteint} for the multiplier ideal under finite maps may fail in positive characteristic.

% One can define the multiplier ideal in any (including mixed) characteristic in the following way (even without resolution of singularities).  While multiplier ideals have been defined primarily in the characteristic zero setting \cf \cite{LazarsfeldPositivity2}, one facet of their pre-history was defined for any normal integral scheme -- Lipman's adjoint ideals.  The definition we give here (which makes sense in any characteristic) is a slight generalization of Lipman's definition to the modern setting of pairs.
\renewcommand{\J}{\mathcal{J}}
% \begin{definition}  \cite{LipmanAdjointsOfIdeals,LazarsfeldPositivity2}
% \label{def:multiplierideals}
% Suppose that $A$ is a normal ring and $X = \Spec A$.
% Further suppose that $(X; \Delta)$ is a pair such that $K_X + \Delta$ is $\bQ$-Cartier.
% The multiplier ideal $\J(X; \Delta)$ is given by
% \[
% \J(A; \Delta) = \bigcap_{\theta \: Z \to X} \Gamma\left(X, \theta_{*}\O_{Z}( \lceil K_{Z} - \theta^{*} (K_{X} + \Delta)\rceil)\right)
% \]
% where the intersection is taken inside $K(A)$ over the set of all normal proper birational modifications $\theta \: Z \to X$.
% \end{definition}

Suppose that $A$ is a normal ring in arbitrary characteristic and that $(X = \Spec A, \Delta)$ is an affine pair such that $K_X + \Delta$ is $\bQ$-Cartier.  As discussed in \autoref{sec.MultiplierRational}, $\J(X;\Delta)$ is only known to be quasi-coherent assuming a theory of resolution of singularities is at hand.  Nonetheless, it is always a sheaf of (fractional) ideals, and in this section we will use the notation $\J(A; \Delta) := H^{0}(X,\J(X; \Delta)) $ to denote the corresponding ideal of global sections.

% \begin{remark}
% As defined, the multiplier ideal is only an ideal and not a sheaf.  We do this because in general, it is not clear that the corresponding intersection of sheaves is coherent.  In particular, we use the notation $\J(A; \Delta)$ instead of $\J(X; \Delta)$ to distinguish these two objects.
% \end{remark}

% Given a sufficiently nice theory of resolution of singularities, it is easy to see that the intersection in the definition may be omitted; instead, a single log resolution $\theta \: Z \to X$ of $(X; \Delta)$ would suffice.  Specifically, one should ask not only that a log resolution of $(X; \Delta)$ exist, but further that any normal proper birational modification can be dominated by a log resolution. See \cite{LazarsfeldPositivity2} for details.

An important and useful perspective, largely in the spirit of \cite{LipmanAdjointsOfIdeals}, is to view
\autoref{def:multiplierideals} as a collection of valuative conditions for membership in the multiplier ideal $\J(A; \Delta)$.
Specifically, suppose $E$ is a prime divisor on a normal proper birational modification $\theta \: Z \to X$.  After identification of the function fields $K =\Frac(A) = K(X) = K(Z)$, $E$ gives rise to a valuation $\ord_{E}$ centered on $X$.  The valuation ring of $\ord_{E}$ is simply the local ring $\O_{Z,E}$. Thus, we have that $\J(A;\Delta)$ can be described as the fractional ideal inside of $K(X)$ given by

\begin{align*}
\J(A;\Delta) &= \displaystyle\bigcap_{\substack{\theta \: Z \to X \\ \text{Prime } E \text{ on } Z}}
\O_{Z,E}(\lceil K_{Z} - \theta^{*}(K_{X} + \Delta)\rceil)\\
&= \left\{ \, f \in K \, \, \left| \, \, \begin{array}{c}\ord_{E}(f) \geq \ord_{E}(\lfloor \theta^{*}(K_{X} + \Delta) -K_{Z} \rfloor) \medskip \\  \text{for all $\theta \: Z \to X$ and all prime $E$ on $Z$} \end{array} \right. \, \right\}
\end{align*}

% In order to explore transformation rules for the multiplier ideal via alterations, in any characteristic, we point out two potential descriptions of their behavior.  Given a finite (separable) surjective map of normal varieties $\pi \: \Spec B = Y \to X = \Spec A$ with $\Delta$ as above on $X$. Set $\Delta_Y = f^* \Delta_X - \Ram_\pi$ where $\Ram_{\pi}$ is the ramification divisor. Then one might expect that either:
% \begin{itemize}
%  \item[(1)]  $\J(B; \Delta_Y) \cap K(A) = \J(A; \Delta_X)$, or
%  \item[(2)]  if $\Tr : K(B) \to K(A)$ is the trace map, then $\Tr(\pi_{*}\J(B; \Delta_{Y})) = \J(A; \Delta_{X})$.
% \end{itemize}
% Noting that (1) is the standard transformation rule for multiplier ideals in characteristic zero, see \cite[Theorem 9.5.42]{LazarsfeldPositivity2} and (2) is the transformation rule for test ideals in characteristic $p > 0$ taken from \cite{SchwedeTuckerTestIdealFiniteMaps}.
% However, \cite[Example 6.31]{SchwedeTuckerTestIdealFiniteMaps} provides us with a finite surjective separable map in positive characteristic of normal varieties $\pi \: Y \to X$ such that (1) does not hold.  Furthermore, condition (2) does not hold because of \cite[Example 7.12]{SchwedeTuckerTestIdealFiniteMaps}.

We now show how to generalize the characteristic zero transformation rule for multiplier ideals under finite maps \eqref{eq:multfiniteint} so as to incorporate the trace map as in \eqref{eq:multfinitetrace}.
Furthermore, by working in arbitrary characteristic, we also recover both of these transformation rules in positive characteristic for separable finite maps of degree prime to the characteristic.

\begin{theorem}
\label{thm:BehaviorOfTheMultiplierIdealViaTrace}
Suppose that $\pi \: \Spec B = Y \to X = \Spec A$ is a finite dominant map of normal integral schemes of any characteristic and that $(X, \Delta_X)$ is a pair such that $K_X + \Delta_X$ is $\bQ$-Cartier.  Define $\Delta_Y = \Delta_X - \Ram_{\pi}$. Then
\[
\Tr(\J(B; \Delta_{Y})) \subseteq \J(A; \Delta_{X}) \, \, .
\]
 Furthermore, if the field trace map $\Tr \: \Frac(B) \to \Frac(A)$ satisfies $\Tr(B) = A$ (\textit{e.g.} the degree of $\pi$ is prime to the characteristic), then
\[
\Tr(\J(B; \Delta_Y)) = \J(B; \Delta_{Y}) \cap K(A) = \J(A; \Delta_X) .
\]
\end{theorem}
\begin{proof}
Suppose $f \in \J(B; \Delta_{Y})$.  Fix a prime divisor $E$ on a normal proper birational modification $\theta \: Z \to X$.  Consider the discrete valuation ring $R = \O_{Z,E}$, viewed as a subring of $K(A)$, and let $r \in K(A)$ be a uniformizer for $R$.  Denote by $S$ the integral closure of $R$ inside of $K(B)$.  Then $S$ can also be realized in the following manner.  Let $W$ be the normal scheme fitting into a commutative diagram
\[  \xymatrix{
W    \ar[d]_{\rho}    \ar[r]^{\eta}    &     Y   \ar[d]^{\pi}   \\
Z    \ar[r]_{\theta}                 &     X
} \]
where $\rho$ is finite and $\eta$ is birational (that is, take $W$ to be the normalization of the relevant irreducible component of $Y \times_X Z$).  Let $E_{1}, \ldots, E_{k}$ be the prime divisors on $W$ mapping onto $E$.  Then we have $S = \bigcap_{i=1}^{k} \O_{W,E_{i}}$, where again we have considered each $\O_{W, E_{i}}$ as a subring of $K(B)$.  In particular, for $g \in K(B)$, we have $g \in S$ if and only if $\ord_{E_{i}}(g) \geq 0$ for all $i = 1, \ldots, k$.

Let $\Phi$ be a generator for the rank one free $S$-module $\Hom_{R}(S,R)$.  If we write $\Tr \: S \to R$ as $\Tr(\blank) = \Phi(c \cdot \blank)$, we know from \cite[Proposition 4.8]{SchwedeTuckerTestIdealFiniteMaps} that $\divisor_{W}(c) = \Ram_{\rho}$ so that
\[
\ord_{E_{i}}(c) = \ord_{E_{i}}(K_{W} - \rho^{*}K_{Z}) \, \, .
\]
Let $\lambda_{E} = \ord_{E}(\lceil K_{Z} - \theta^{*}(K_{X} + \Delta_{X}) \rceil)$ and consider $g = cr^{\lambda_{E}}f$.  Since $f \in \J(B; \Delta_{Y})$, it follows that $\ord_{E_{i}}(f) + \lceil K_{W} - \eta^{*}(K_{Y} + \Delta_{Y}) \rceil \geq 0$ and hence
\begin{align*}
\ord_{E_{i}}(g) &\geq  \ord_{E_{i}}( K_{W} - \rho^{*}K_{Z} + \rho^{*}\lceil K_{Z} - \theta^{*}(K_{X} + \Delta_{X}) \rceil - \lceil K_{W} - \eta^{*}(K_{Y} + \Delta_{Y}) \rceil) \\
& =  \ord_{E_{i}}(\rho^{*} \lceil - \theta^{*}(K_{X} + \Delta_{X}) \rceil - \lceil - \eta^{*}(K_{Y} + \Delta_Y) \rceil) \\
& =  \ord_{E_{i}}(\rho^{*} \lceil - \theta^{*}(K_{X} + \Delta_{X}) \rceil - \lceil - \eta^{*}(K_{Y} + \pi^{*}\Delta_{X} - (K_{Y} - \pi^{*}K_{X})) \rceil) \\
& =  \ord_{E_{i}}( \rho^{*} \lceil - \theta^{*}(K_{X} + \Delta_{X}) \rceil - \lceil \rho^{*} ( - \theta^{*}(K_{X} + \Delta_{X}))\rceil) \\
& \geq  0.
\end{align*}
It now follows that $g \in S$, and thus $\Phi(g) = r^{\lambda_{E}}\Tr(f) \in R$.  In other words, we have
\begin{align*}
  \ord_{E}(\Phi(g)) & =  \ord_{E}(\Tr(f)) + \ord_{E}(r^{\lambda_{E}}) \\
& =  \ord_{E}(\Tr(f)) + \ord_{E}( \lceil K_{Z} - \theta^{*}(K_{X} + \Delta_{X} \rceil)    \geq 0
\end{align*}
and we conclude that $\Tr(f) \in \J(A; \Delta_{X})$ and thus $\Tr(\J(B; \Delta_{Y})) \subseteq \J(B; \Delta_{X})$ as desired.

Note that every divisorial valuation $\nu \: K(B) \setminus \{0\} \to \Z$ centered on $Y$ can be realized as $\nu = \ord_{E_{i}}$ as in the setup above.  Indeed, the restriction $\nu$ to $K(A)$ gives rise to a discrete valuation ring whose residue field has transcendence degree $(\dim(Y) -1) = (\dim(X) - 1)$ over $\Lambda$; see \cite[Chapter VI, Section 8]{Bourbaki1998}.  By Proposition 2.45 in \cite{KollarMori}, this valuation ring can be realized as $\O_{Z,E}$ for some prime divisor $E$ on $\theta \: Z \to X$ as above, so that $\nu = \ord_{E_{i}}$ for some $i$.

Let us now argue that $\J(A; \Delta_{X}) \cdot \pi_* \O_{Y} \subseteq \pi_* \J(B; \Delta_{Y})$.  Suppose $h \in \J(A; \Delta_{X})$.  We have by assumption
\begin{align*}
0 &\leq \ord_{E}(h) + \ord_{E}(\lceil K_{Z} - \theta^{*}(K_{X}+\Delta_{X})\rceil)
\end{align*}
whence it follows from \cite[Chapter IV, Proposition 2.2]{Hartshorne}, that
\begin{align*}
  0 & \leq  \ord_{E_{i}}(h) + \ord_{E_{i}}(\rho^{*}\lceil K_{Z} - \theta^{*}(K_{X} + \Delta_{X}) \rceil) \\
& \leq  \ord_{E_{i}}(h) + \ord_{E_{i}}(\Ram_{\rho}) + \ord_{E_{i}}(\lceil \rho^{*}(K_{Z} - \theta^{*}(K_{X} + \Delta_{X})) \rceil) \\
& =  \ord_{E_{i}}(h) + \ord_{E_{i}}(\lceil K_{W} - \eta^{*}(K_{Y} + \Delta_{Y}) \rceil ) \, .
\end{align*}
Thus, we conclude $h \in \J(B; \Delta_{Y})$ and hence $\J(A; \Delta_{X}) \cdot \pi_* \O_{Y} \subseteq \pi_* \J(B; \Delta_{Y})$.

Now assume in addition that the trace map is surjective, \ie $\Tr(B) = A$.  We have, using the surjectivity of trace at the third inequality, that
\[
\J(A; \Delta_{X}) \subseteq (\J(A; \Delta_{X}) \cdot \pi_* \O_Y) \cap \O_X  \subseteq \pi_* \J(B; \Delta_Y) \cap \O_X \subseteq \Tr\left(\pi_* \J(B; \Delta_{Y})\right) \subseteq \J(A; \Delta_{X}).
\]
This necessitates equality throughout, which completes the proof.
\end{proof}

We now complete the proof of our main theorem.

\begin{corollary}
\label{cor.MultIdealViaAlterations}
 Suppose that $(X, \Delta)$ is a pair in characteristic zero.  Then
\[
\J(X; \Delta) = \bigcap_{\pi : Y \to X} \Image \Big( \pi_* \O_Y(\lceil K_Y - \pi^*(K_X + \Delta) \rceil ) \to[\tr] K(X)\Big)
\]
where $\pi$ ranges over all alterations with $Y$ normal, $\tr \: K(Y) \to K(X)$ is the field trace, and we have that $K_{Y} = \pi^{*}K_{X} + \Ram_{\pi}$ wherever $\pi$ is finite.
\end{corollary}
\begin{proof}
Because of the existence of resolution of singularities in characteristic zero, we may assume that each $Y$ is smooth and that $K_Y - \pi^*(K_X + \Delta)$ is supported on a simple normal crossings divisor on $Y$.
 First consider a finite map $f \: W \to X$ with $W$ normal (note that we are in characteristic zero, so the map is separable).  If we pick a representative $K_X$ such that $\O_X(K_X) = \omega_X$ and also pick $K_W = f^* K_X + \Ram_f$, then we recall from \autoref{ex.traceForFinite} that the field trace
$ \Tr \: K(W) \to K(X)$
induces a map $\Tr \: \O_W(K_W) \to \O_X(K_X)$ which is identified with the Grothendieck trace $\Tr_f \: f_* \omega_W \to \omega_X$.

Fix any proper dominant map $\pi \: Y \to X$.  Factor $\pi$ through a birational map $\rho \: Y \to W$ and a finite map $f \: W \to X$ with $W$ normal via Stein factorization.  Thus $\pi_* \omega_Y \to \omega_X$ factors through $\Tr \: f_* \omega_W \to \omega_X$.  It follows that
\[
 \pi_* \O_Y\left(K_Y - \pi^*(K_X + \Delta) \right) \to K(X)
\]
factors through $\Tr \: K(W) \to K(X)$.  Furthermore, $\rho_* \O_Y(K_Y - \pi^*(K_X + \Delta) )$ is by definition the multiplier ideal $\J(W; \Delta_W)$ where $\Delta_W = f^* \Delta - \Ram_f$, since $Y$ is smooth.  Thus, $\Image \Big( \pi_* \O_Y(K_Y - \pi^*(K_X + \Delta) ) \to K(X)\Big)$ is simply $\Tr(\rho_{*}\J(W; \Delta_{W}))$.  But that is equal to $\mJ(X; \Delta)$ by \autoref{thm:BehaviorOfTheMultiplierIdealViaTrace}.
\end{proof}

Finally, we turn our attention to the behavior of the multiplier ideal under proper dominant maps in characteristic zero.

\begin{theorem}
\label{thm.TransformationRuleForMultProperDominant}
 If $(X, \Delta)$ is a log-$\bQ$-Gorenstein pair in characteristic zero, then
\[
 \J(X; \Delta) = \bigcap_{\pi : Y \to X} \Image \Big( \myH^{\dim Y - \dim X} \myR \pi_* \O_Y(\lceil K_Y - \pi^*(K_X + \Delta) \rceil ) \to[\tr_{\pi}] K(X)\Big)
\]
where the intersection runs over all proper dominant maps from normal varieties $Y$.  Note the map to $K(X)$ is induced from the trace map as in \autoref{prop.TraceInGenerality}.
\end{theorem}
\begin{proof}
We may restrict our maps $\pi : Y \to X$ to those maps where $Y$ is regular and which factor through a fixed regular alteration $f : Z \to X$ such that $f^* (K_X + \Delta)$ is an integral Cartier divisor.  Thus by \autoref{cor.MultIdealViaAlterations}, it is sufficient to show that
\[
 \myH^{\dim Y - \dim X} \myR (f \circ \rho)_* \O_Y(K_Y - \pi^*(K_X + \Delta) ) \to  \myH^0 \myR f_* \omega_Z(-f^* (K_X + \Delta))
\]
is surjective on global sections.  However, by the statement of \cite[Theorem 2]{KovacsRat}, first correctly proved in full generality in \cite[Theorem 4.1.3]{BhattThesis}, the natural map $\O_Z \to \myR \rho_* \O_Y$ has a left inverse in the derived category, and thus so does
\[
\myR f_* \myR \rho_* \left( \omega_Y[\dim Y] \tensor \O_Y (\pi^* (K_X + \Delta)) \right) \to \myR f_* \left( \omega_Z[\dim Z] \tensor \O_Z(f^* (K_X + \Delta)) \right).
\]
In particular, taking $-d$th cohomology yields the desired surjection on global sections.
\end{proof}

\section{Further questions}%\ok{--}{--}{--}

This theory suggests a large number of potential directions for further inquiry.  We highlight a few of the more obvious ones below.

\begin{question} [Mixed characteristic]
What can be said in mixed characteristic?  In particular, does the intersection from our Main Theorem stabilize for schemes in mixed characteristic?
\end{question}

We have learned that M. Hochster and W. Zhang have made progress on this question in low dimensions for isolated singularities.

\begin{question}[Adjoint ideals]
Can one develop a characteristic theory analogous to the theory of adjoint ideals, \cf \cite[9.3.E]{LazarsfeldPositivity2} or \cite{TakagiPLTAdjoint}, described via alterations or finite covers?
\end{question}

The characterization of test ideals, as well as $F$-regular and $F$-rational singularities suggests the following:

\begin{question}[$F$-pure singularities]
Can $F$-pure (or $F$-injective) singularities likewise be described by alterations?
\end{question}

Finally, we consider the following question:

\begin{question}[Effectivity of covers and alterations]
Given a pair $(X, \Delta)$, how can we identify finite covers (or alterations) $\pi : Y \to X$ such that
\[
\tau(X; \Delta) = \Image \big( \lceil \pi_* \O_Y(K_Y - \pi^*(K_X + \Delta) \rceil ) \to[\tr_{\pi}] K(X)\big)?
\]
In other words, can we determine when the intersection from our Main Theorem stabilizes?
\end{question}

We do not have a good answer to this question.  The key point in our construction is repeated use of the Equational Lemma \cite{HochsterHunekeInfiniteIntegralExtensionsAndBigCM, HunekeLyubeznikAbsoluteIntegralClosure,sannai_galois_2011}.  The procedure in that Lemma is constructive.  However, this is not very satisfying.  It would be very satisfying and likely useful if one had a different geometric or homological criterion for identifying $\pi : Y \to X$ as in the question above.

%\begin{question}[Triples] Can one generalize the results of this paper to the context of triples $(X, \Delta, \ba^t)$? \end{question}

\bibliographystyle{skalpha}%\ok{--}{--}{--}

\bibliography{CommonBib}
\end{document}